\documentclass[12pt,reqno]{amsart}%
\usepackage{amsmath, amsfonts, amssymb, amsthm, amscd, amsbsy}
\usepackage{fancyhdr}
\usepackage[usenames,dvipsnames,svgnames,x11names,hyperref]{xcolor}
\usepackage{geometry}
\usepackage{graphicx}
\usepackage{hyperref}
\usepackage{amsmath}
\usepackage{amsfonts}
\usepackage{amssymb}
\usepackage{lineno}%
\providecommand{\U}[1]{\protect\rule{.1in}{.1in}}
\hypersetup{colorlinks,breaklinks,
	linkcolor={Fuchsia},
	citecolor={ForestGreen},
	urlcolor={NavyBlue}}
\newtheorem{theorem}{Theorem}[section]
\theoremstyle{plain}

\newtheorem{proposition}{Proposition}[section]

\numberwithin{equation}{section}
\allowdisplaybreaks
\newtheorem{theorema}{Theorem}[section]

\begin{document}
\title[Stability of Gaussian Poincar\'{e} and HUP monomial weight]{Stability of Gaussian Poincar\'{e} inequalities and Heisenberg Uncertainty
Principle with monimial weights
}
\author{Nguyen Lam}
\address{Nguyen Lam: School of Science and the Environment, Grenfell Campus, Memorial
University of Newfoundland, Corner Brook, NL A2H5G4, Canada}
\email{nlam@mun.ca}
\author{Guozhen Lu}
\address{Guozhen Lu: Department of Mathematics, University of Connecticut, Storrs, CT
06269, USA}
\email{guozhen.lu@uconn.edu}
\author{Andrey Russanov}
\address{Andrey Russanov: Department of Mathematics, University of Connecticut, Storrs,
CT 06269, USA}
\email{andrey.russanov@uconn.edu}

\begin{abstract}
We use the Bakry-\'{E}mery curvature-dimension criterion and $\Gamma$-calculus
to establish the Poincar\'{e} inequality with monomial Gaussian measure, and
then apply the duality approach to study its improvements and its gradient stability. We
also set up the scale-dependent Poincar\'{e} inequality with monomial Gaussian
type measure and use it to inspect the stability of the Heisenberg Uncertainty
Principle with monomial weight. Finally, we apply the improved versions of the
monomial Gaussian Poincar\'{e} inequality to investigate the improved
stability of the Heisenberg Uncertainty Principle with monomial weight. As
special cases of our main results, we obtain the
gradient stability of the classical Gaussian Poincar\'{e} inequality, which is of independent interest. Moreover, we also establish the stability of the sharp stability inequality of the classical Heisenberg Uncertainty
Principle proved in \cite{CFLL24}.

\end{abstract}
\subjclass[2010]{}

\thanks{The first author was partially supported by an NSERC Discovery Grant. The research for the second and third authors were partially supported by collaboration grants and a Simons Fellowship in Mathematics from the Simons Foundation.}
\maketitle
\section{Introduction}

The main goal of this paper is to set up the Poincar\'{e} inequality with
monomial Gaussian measure, its improvements and its gradient stability, and use them to
investigate the stability and the improved stability of the Heisenberg
Uncertainty Principle with monomial weight. Our main motivation comes from
\cite{CFLL24}, in which the authors used the classical Poincar\'{e} inequality
with Gaussian weight and the Heisenberg Uncertainty Principle identity to
derive the sharp stability, with explicit optimal constants, of the Heisenberg
Uncertainty Principle.

The question about the stability of geometric and functional inequalities was first raised by
Brezis and Lieb in \cite{BL85}. More clearly, Brezis and Lieb asked in
\cite{BL85} whether the difference of the two terms in the Sobolev
inequalities controls the distance to the family of extremal functions. Brezis
and Lieb's question was answered affirmatively by Bianchi and Egnell in
\cite{BE91}. Indeed, by exploiting the special structure of the Hibert space
$W^{1,2}\left(  \mathbb{R}^{N}\right)  $, Bianchi and Egnell established in
\cite{BE91} that
\[
\int_{\mathbb{R}^{N}}\left\vert \nabla u\right\vert ^{2}dx-S_{N}\left(
\int_{\mathbb{R}^{N}}|u|^{\frac{2N}{N-2}}dx\right)  ^{\frac{N-2}{N}}\geq
c_{BE}\inf_{U\in E_{Sob}}\int_{\mathbb{R}^{N}}\left\vert \nabla\left(
u-U\right)  \right\vert ^{2}dx
\]
for some stability constant $c_{BE}>0$. Here $S_{N}$ is the sharp Sobolev
constant and $E_{Sob}$ is the manifold of the optimizers of the Sobolev
inequality. Brezis and Lieb's question together with Bianchi and Egnell's
result have initiated the program of studying the quantitative stability
results for functional and geometric inequalities that has attracted great attention. The literature on the topic
is extremely vast and therefore, we just refer the interested reader to
\cite{BWW03, BDNN20, CF13, CFW13, CLT23, CFMP09, FJ1, FJ2, FN19, FZ22, LW99},
for the study of stability of Sobolev type inequalities.

It is worth noting that the stability constants and whether or not the
stability inequalities can be attained have usually not been investigated in
the literature. For instance, the precise information on the stability constant
$c_{BE}$ was totally missing in the literature until very recently. In a very
recent paper \cite{DEFFL}, Dolbeault, Esteban, Figalli, Frank and Loss shed
some light on the stability constant $c_{BE}$ by providing some optimal lower bounds for $c_{BE}$ when the dimension $N\to \infty$, and established the stability for Gaussian
log-Sobolev inequality as an application. Also, K\"{o}nig proved in
\cite{Kon23} that the optimal lower bound is strictly smaller than the
spectral gap constant $\frac{4}{N+4}$, and derived in \cite{Kon22} its
attainability. 
More recently, in \cite{CLT24}, Chen, Tang and the second
author studied the explicit lower bounds of the stability of the
Hardy-Littlewood-Sobolev inequalities through which they also deduced some explicit lower bounds
for the stability inequality of the higher and fractional order Sobolev
inequalities. Moreover, they also obtained
in  \cite{CLT242, CLT243} the optimal asymptotic lower bounds for the Hardy-Littlewood-Sobolev inequalities and higher and fractional Sobolev inequalities when $N\rightarrow
\infty$  for $0<s<N/2$ and when $s\rightarrow0$ for all $N$.  This latter estimate when $s\to 0$  also allows them to derive the global stability for the log-Sobolev inequality on the sphere
established by Beckner \cite{Beckner93} and sharpen the
earlier local stability obtained by Chen et al in \cite{CLT23}.

\medskip

Heisenberg Uncertainty Principle (HUP) is a fundamental concept in quantum
mechanics that states that there is a limit to the precision with which
certain pairs of physical properties, such as position and momentum, can be
simultaneously known. Mathematically, it can be formulated as follows: For
$u\in S_{0}$, one has
\begin{equation}
\left(  \int_{\mathbb{R}^{N}}\left\vert \nabla u\right\vert ^{2}dx\right)
\left(  \int_{\mathbb{R}^{N}}|u|^{2}|x|^{2}dx\right)  \geq\frac{N^{2}}
{4}\left(  \int_{\mathbb{R}^{N}}|u|^{2}dx\right)  ^{2}. \label{HUP}%
\end{equation}
Here $S_{0}$ is the completion of $C_{0}^{\infty}\left(  \mathbb{R}
^{N}\right)  $ under the norm $\left(  \int_{\mathbb{R}^{N}}\left\vert \nabla
u\right\vert ^{2}dx\right)  ^{\frac{1}{2}}+\left(  \int_{\mathbb{R}^{N}
}|u|^{2}|x|^{2}dx\right)  ^{\frac{1}{2}}$. In \cite{CFLL24}, the authors
proved the following HUP identity: For $u\in S_{0}$, $u\neq0$ and
$\lambda=\left(  \frac{\int_{\mathbb{R}^{N}}|u|^{2}|x|^{2}dx}{\int
_{\mathbb{R}^{N}}|\nabla u|^{2}dx}\right)  ^{\frac{1}{4}}$, one has
\begin{equation}
\left(  \int_{\mathbb{R}^{N}}\left\vert \nabla u\right\vert ^{2}dx\right)
^{\frac{1}{2}}\left(  \int_{\mathbb{R}^{N}}|u|^{2}|x|^{2}dx\right)  ^{\frac
{1}{2}}-\frac{N}{2}\int_{\mathbb{R}^{N}}|u|^{2}dx=\frac{\lambda^{2}}{2}
\int_{\mathbb{R}^{N}}\left\vert \nabla\left(  ue^{\frac{|x|^{2}}{2\lambda^{2}
}}\right)  \right\vert ^{2}e^{-\frac{|x|^{2}}{\lambda^{2}}}dx. \label{HUPI}%
\end{equation}
Obviously, HUP identity (\ref{HUPI}) provides several important information
about the HUP. For instance, it can be deduced from (\ref{HUPI}) that all
optimizers for (\ref{HUP}) are the classical Gaussian profiles. Let
$E_{HUP}=\left\{  \alpha e^{-\beta\left\vert x\right\vert ^{2}}:\alpha
\in\mathbb{R}\text{, }\beta>0\right\}  $ be the set of all optimizers for
(\ref{HUP}). Motivated by Brezis and Lieb's question, we could ask here 
whether or not the HUP is stable. More clearly, is it true that $\delta\left(
u\right)  \approx0$ implies $u\approx\alpha e^{-\beta\left\vert x\right\vert
^{2}}$ in some sense for some $\alpha\in\mathbb{R}$, $\beta>0$, and for some
Heisenberg deficit $\delta$?

In an effort to answer the question of the stability of the HUP, McCurdy
and Venkatraman applied the concentration-compactness arguments and proved in
\cite{MV21} that there exist universal constants $C_{1}>0$ and $C_{2}(N)>0$
such that
\[
\delta_{2}(u)\geq C_{1}\left(  \int_{\mathbb{R}^{N}}|u|^{2}dx\right)
d_{1}^{2}(u,E_{HUP})+C_{2}(N)d_{1}^{4}(u,E_{HUP}),
\]
for all $u\in S_{0}$. Here
\[
\delta_{2}\left(  u\right)  :=\left(  \int_{\mathbb{R}^{N}}|\nabla
u|^{2}dx\right)  \left(  \int_{\mathbb{R}^{N}}|x|^{2}|u|^{2}dx\right)
-\dfrac{N^{2}}{4}\left(  \int_{\mathbb{R}^{N}}|u|^{2}dx\right)  ^{2}
\]
is a HUP deficit and $d_{1}(u,A):=\inf_{v\in A}\left\{  \left\Vert
u-v\right\Vert _{2}\right\}  $ is the distance from $u$ to the set $A$.
Therefore, $\delta\left(  u\right)  \approx0$ implies $u\approx\alpha
e^{-\beta\left\vert x\right\vert ^{2}}$in $L^{2}\left(  \mathbb{R}^{N}\right)
$ for some $\alpha\in\mathbb{R}$, $\beta>0$. A simple and constructive proof
was provided later by Fathi in \cite{F21} to show that $C_{1}=\frac{1}{4}$ and
$C_{2}=\frac{1}{16}$. However, these constants are not sharp. Eventually, the
authors in \cite{CFLL24} combined the HUP identity (\ref{HUPI}) and the
following Poincar\'{e} inequality with Gaussian type measure: for all
$\lambda\neq0$,
\[
\int_{\mathbb{R}^{N}}|\nabla u|^{2}e^{-\frac{1}{2\left\vert \lambda\right\vert
^{2}}|x|^{2}}dx\geq\frac{1}{\left\vert \lambda\right\vert ^{2}}\inf_{c}
\int_{\mathbb{R}^{N}}\left\vert u-c\right\vert ^{2}e^{-\frac{1}{2\left\vert
\lambda\right\vert ^{2}}|x|^{2}}dx\text{, }
\]
to show the following sharp stability of HUP:

\begin{theorema}
\label{A}For all $u\in S_{0}:$
\begin{equation}
\delta_{1}\left(  u\right)  :=\left(  \int_{\mathbb{R}^{N}}\left\vert \nabla
u\right\vert ^{2}dx\right)  ^{\frac{1}{2}}\left(  \int_{\mathbb{R}^{N}}
|u|^{2}|x|^{2}dx\right)  ^{\frac{1}{2}}-\frac{N}{2}\int_{\mathbb{R}^{N}
}|u|^{2}dx\geq d_{1}^{2}(u,E_{HUP}). \label{SHUP}%
\end{equation}
Moreover, the inequality is sharp and the equality can be attained by
nontrivial functions $u\notin E_{HUP}$.
\end{theorema}

As a consequence, we can deduce from the above Theorem that
\[
\delta_{2}(u)\geq N\left(  \int_{\mathbb{R}^{N}}|u|^{2}dx\right)  d_{1}
^{2}(u,E_{HUP})+d_{1}^{4}(u,E_{HUP})
\]
and the inequality is sharp and can be attained by nontrivial functions
$u\notin E_{HUP}$.

Since the inequality (\ref{SHUP}) in Theorem \ref{A} can be attained by
nontrivial optimizers, we can once again ask for its stability. More clearly,
let $F_{HUP}$ be the set of all extremizers of (\ref{SHUP}). Then it will be interesting to ask the following
\medskip

\noindent\textbf{Question 1.} \textit{Does one have that }
\[
\delta_{1}\left(  u\right)  -d_{1}^{2}(u,E_{HUP})\gtrsim d_{2}^{2}(u,F_{HUP})
\]
\textit{for some distance function }$d_{2}(u,F_{HUP})$\textit{ from }
$u$\textit{ to }$F_{HUP}$\textit{?}

It is also worthy to note that the HUP identity (\ref{HUPI}) is just a
consequence of a more general $L^{2}$-Caffarelli-Kohn-Nirenberg identity that
has been established in \cite{CFLL24}. Indeed, it was showed in \cite{CFLL24} that

\begin{theorema}
\label{B}\textit{Let }$0<R\leq\infty$\textit{, }$U$\textit{ and }$V$\textit{
be }$C^{1}$-\textit{functions on }$\left(  0,R\right)  $ and let
\[
W\left(  r\right)  =\left(  U\left(  r\right)  V\left(  r\right)  \right)
^{\prime}+\left(  N-1\right)  \frac{U\left(  r\right)  V\left(  r\right)  }
{r}-V^{2}\left(  r\right)  \text{.}
\]
Then for all $u\in C_{0}^{\infty}\left(  B_{R}\setminus\left\{  0\right\}
\right)  $, we have
\begin{align*}
&  \left(  {\int\limits_{B_{R}}}U^{2}\left(  \left\vert x\right\vert \right)
\left\vert \nabla u\right\vert ^{2}dx\right)  ^{\frac{1}{2}}\left(
{\int\limits_{B_{R}}}V^{2}\left(  \left\vert x\right\vert \right)  \left\vert
u\right\vert ^{2}dx\right)  ^{\frac{1}{2}}\\
&  =\frac{1}{2}{\int\limits_{B_{R}}}\left[  W\left(  \left\vert x\right\vert
\right)  +V^{2}\left(  \left\vert x\right\vert \right)  \right]  \left\vert
u\right\vert ^{2}dx\\
&  +\frac{1}{2}{\int\limits_{B_{R}}}\left\vert \frac{\left\Vert Vu\right\Vert
_{2}^{\frac{1}{2}}}{\left\Vert U\left\vert \nabla u\right\vert \right\Vert
_{2}^{\frac{1}{2}}}U\left(  \left\vert x\right\vert \right)  \nabla
u+\frac{\left\Vert U\left\vert \nabla u\right\vert \right\Vert _{2}^{\frac
{1}{2}}}{\left\Vert Vu\right\Vert _{2}^{\frac{1}{2}}}V\left(  \left\vert
x\right\vert \right)  u\frac{x}{\left\vert x\right\vert }\right\vert ^{2}dx.
\end{align*}

\end{theorema}

In particular, by choosing $U=sign\left(  N-a-b-1\right)  r^{-b}$, $V=r^{-a}
$. Then
\[
W=\left[  \left\vert N-1-a-b\right\vert r^{-a-b-1}-r^{-2a}\right]  .
\]
Therefore, we get
\begin{align*}
&  \left(  {\int\limits_{\mathbb{R}^{N}}}\frac{1}{\left\vert x\right\vert
^{2b}}\left\vert \nabla u\right\vert ^{2}dx\right)  ^{\frac{1}{2}}\left(
{\int\limits_{\mathbb{R}^{N}}}\frac{1}{\left\vert x\right\vert ^{2a}
}\left\vert u\right\vert ^{2}dx\right)  ^{\frac{1}{2}}\\
&  =\left\vert \frac{N-1-a-b}{2}\right\vert {\int\limits_{\mathbb{R}^{N}}
}\frac{1}{\left\vert x\right\vert ^{a+b+1}}\left\vert u\right\vert ^{2}dx\\
&  +\frac{1}{2}{\int\limits_{\mathbb{R}^{N}}}\left\vert sign\left(
N-a-b-1\right)  \frac{\left\Vert \frac{u}{\left\vert x\right\vert ^{a}
}\right\Vert _{2}^{\frac{1}{2}}}{\left\Vert \frac{\left\vert \nabla
u\right\vert }{\left\vert x\right\vert ^{b}}\right\Vert _{2}^{\frac{1}{2}}
}\frac{1}{\left\vert x\right\vert ^{b}}\nabla u+\frac{\left\Vert
\frac{\left\vert \nabla u\right\vert }{\left\vert x\right\vert ^{b}
}\right\Vert _{2}^{\frac{1}{2}}}{\left\Vert \frac{u}{\left\vert x\right\vert
^{a}}\right\Vert _{2}^{\frac{1}{2}}}\frac{1}{\left\vert x\right\vert ^{a}
}u\frac{x}{\left\vert x\right\vert }\right\vert ^{2}dx.
\end{align*}
Furthermore, with $a=-1$ and $b=0$, we obtain (\ref{HUPI}).

Theorem \ref{B} has been extended further in \cite{AJLL23}. Indeed, it was
showed in \cite{AJLL23} that

\begin{theorema}
\label{C}Let $\Omega$ be an open set in $\mathbb{R}^{N}$, $V\geq0$ be a smooth
function and $\overrightarrow{X}\in C^{1}\left(  \Omega,\mathbb{R}
^{N}\right)  $. Then for any $u\in C_{0}^{1}\left(  \Omega\right)  $, we have
\begin{align*}
&  \left(  {\int\limits_{\Omega}}V\left\vert \nabla u\right\vert
^{2}dx\right)  ^{\frac{1}{2}}\left(  {\int\limits_{\Omega}}V\left\vert
\overrightarrow{X}\right\vert ^{2}\left\vert u\right\vert ^{2}dx\right)
^{\frac{1}{2}}+\frac{1}{2}{\int\limits_{\Omega}\operatorname{div}}\left(
V\overrightarrow{X}\right)  \left\vert u\right\vert ^{2}dx\\
&  =\frac{1}{2}{\int\limits_{\Omega}V}\left\vert \left(  \frac{{\int
\limits_{\Omega}}V\left\vert \overrightarrow{X}\right\vert ^{2}\left\vert
u\right\vert ^{2}dx}{{\int\limits_{\Omega}}V\left\vert \nabla u\right\vert
^{2}dx}\right)  ^{\frac{1}{4}}\nabla u-\left(  \frac{{\int\limits_{\Omega}
}V\left\vert \nabla u\right\vert ^{2}dx}{{\int\limits_{\Omega}}V\left\vert
\overrightarrow{X}\right\vert ^{2}\left\vert u\right\vert ^{2}dx}\right)
^{\frac{1}{4}}u\overrightarrow{X}\right\vert ^{2}dx.
\end{align*}

\end{theorema}

We note that a $L^{p}$-version of Theorem \ref{C} has also been studied in
\cite{AJLL23} and then has been used to investigate the stability of the
$L^{p}$-Caffarelli-Kohn-Nirenberg inequalities. Variants of Theorem \ref{C}
have also been studied in \cite{DLL22, FLL22, FLLM23, HY22, LLZ19, LLZ20}, to
name just a few.

\medskip

In this paper, we are interested in the sharp versions and the stability of
the Gaussian Poincar\'{e} inequality and the HUP with monomial weight. Here, a
monomial weight is a weight of the form $x^{A}:=x_{1}^{\alpha_{1}}
...x_{N}^{\alpha_{N}}$, $\alpha_{1},...,\alpha_{N}\geq0$. We also say that the
monomial weight $x^{A}$ is full if $\alpha_{i}>0$ for all $i=1,...,N$.
Otherwise, we say that $x^{A}$ is partial. Let $D=N+\alpha_{1}+...+\alpha_{N}$
and $\mathbb{R}_{\ast}^{N}=\left\{  x\in\mathbb{R}^{N}:x_{i}>0\text{ if
}\alpha_{i}>0\right\}  $ be the dimension and the Weyl chamber associated to
$x^{A}$. We note that the functional and geometric inequalities with monomial
weights have been studied intensively in the literature. For instance,
motivated by an open question raised by Brezis \cite{Brezis1}, Cabr\'{e} and
Ros-Oton established in \cite{CR} the Sobolev inequality with monomial weight
and used it to investigate the problem of the regularity of stable solutions
to reaction-diffusion problems of double revolution. In \cite{CRO, CRS16}, the
authors also studied the Sobolev, Morrey, Trudinger and isoperimetric
inequalities with monomial weight $x^{A}$ and homogeneous weight. The optimal constants of the Trudinger-Moser inequalities with
monomial weights were computed explicitly in
\cite{DNP19, GH21, Lam2}. In \cite{BGL14}, Bakry, Gentil and Ledoux combined the
stereographic projection and the Curvature-Dimension condition to set up the
sharp Sobolev inequality with monomial weight. Also, mass transport approach
was used to study the sharp constants and optimizers for the
Gagliardo-Nirenberg inequalities and logarithmic Sobolev inequalities with
arbitrary norm and with monomial weights in \cite{Lam19, NVH}. In \cite{Cas},
the author provided a proof for the Hardy-Sobolev-type inequalities with
monomial weights. However, the best constant and the extremals for the
inequalities were not studied there. Recently, sharp $L^{p}$-Hardy
inequalities and optimal Hardy-Sobolev inequalities with monomial weight have
also been studied in \cite{DLL22}. In \cite{Wang21}, the author derived a
double-weighted Hardy-Sobolev inequality and used it to establish the Gross'
type logarithmic Sobolev inequality \cite{Gro75, Gro75b} with monomial weights
using product structure and the Moser-Onofri-Beckner inequality
\cite{Beckner93} with monomial weights.

It is also worth mentioning that by choosing suitable potential $V$ and vector
field $\overrightarrow{X}$ in Theorem \ref{C}, we can derive a HUP identity
with monomial weight. Indeed, with $V=x^{A}$ and $\overrightarrow{X}=-x$, we
have
\[
{\operatorname{div}}\left(  V\overrightarrow{X}\right)  =-{\operatorname{div}
}\left(  x^{A}x\right)  =-x^{A}\operatorname{div}\left(  x\right)
-x\cdot\nabla x^{A}=-Dx^{A}.
\]
Therefore, Theorem \ref{C} implies that

\begin{proposition}
\label{P0}For $u\in S_{A}$, $u\neq0$ and $\lambda=\left(  \frac{\int
_{\mathbb{R}_{\ast}^{N}}|u|^{2}|x|^{2}x^{A}dx}{\int_{\mathbb{R}_{\ast}^{N}
}|\nabla u|^{2}x^{A}dx}\right)  ^{\frac{1}{4}}$. We have that
\begin{align*}
&  \left(  \int_{\mathbb{R}_{\ast}^{N}}|\nabla u|^{2}x^{A}dx\right)
^{\frac{1}{2}}\left(  \int_{\mathbb{R}_{\ast}^{N}}|u|^{2}|x|^{2}
x^{A}dx\right)  ^{\frac{1}{2}}-\frac{D}{2}\int_{\mathbb{R}_{\ast}^{N}}
|u|^{2}x^{A}dx\\
&  =\frac{\lambda^{2}}{2}\int_{\mathbb{R}_{\ast}^{N}}\left\vert \nabla\left(
ue^{\frac{|x|^{2}}{2\lambda^{2}}}\right)  \right\vert ^{2}e^{-\frac{|x|^{2}
}{\lambda^{2}}}x^{A}dx.
\end{align*}

\end{proposition}

\noindent Here $S_{A}$ is the completion of $N_{\ast}:=\left\{  u\in C_{0}^{\infty
}\left(  \overline{\mathbb{R}_{\ast}^{N}}\right)  :\nabla u\cdot
\overrightarrow{\eta}=0\text{ on }\partial\mathbb{R}_{\ast}^{N}\right\}  $,
where $\overrightarrow{\eta}$ is the outer normal vector of $\mathbb{R}_{\ast
}^{N}$, under the norm $\left(  \int_{\mathbb{R}_{\ast}^{N}}\left\vert \nabla
u\right\vert ^{2}x^{A}dx\right)  ^{\frac{1}{2}}$ $+\left(  \int_{\mathbb{R}
_{\ast}^{N}}|u|^{2}|x|^{2}x^{A}dx\right)  ^{\frac{1}{2}}$.

It is also clear that from the above identity, we obtain the HUP with monomial
weight
\[
\left(  \int_{\mathbb{R}_{\ast}^{N}}|\nabla u|^{2}x^{A}dx\right)  ^{\frac
{1}{2}}\left(  \int_{\mathbb{R}_{\ast}^{N}}|u|^{2}|x|^{2}x^{A}dx\right)
^{\frac{1}{2}}\geq\frac{D}{2}\int_{\mathbb{R}_{\ast}^{N}}|u|^{2}x^{A}dx,
\]
together with its sharp constant $\frac{D}{2}$ and all its optimizers
$E_{HUPA}:=\left\{  \alpha e^{-\beta\left\vert x\right\vert ^{2}}:\alpha
\in\mathbb{R}\text{, }\beta>0\right\}  $. Therefore, once again, we may ask if
the HUP with monomial weight is stable. More precisely, we would like to
answer the following question:

\noindent\textbf{Question 2.} \textit{Does one have that }
\[
\delta_{A}\left(  u\right)  :=\left(  \int_{\mathbb{R}_{\ast}^{N}}|\nabla
u|^{2}x^{A}dx\right)  ^{\frac{1}{2}}\left(  \int_{\mathbb{R}_{\ast}^{N}
}|u|^{2}|x|^{2}x^{A}dx\right)  ^{\frac{1}{2}}-\frac{D}{2}\int_{\mathbb{R}
_{\ast}^{N}}|u|^{2}x^{A}dx\gtrsim d_{A}^{2}(u,E_{HUPA})
\]
\textit{for some distance function }$d_{A}(u,E_{HUPA})$\textit{ from }
$u$\textit{ to }$E_{HUPA}$\textit{?}

Motivated by the results and approaches in \cite{CFLL24}, Question 1 and
Question 2, the first main goal of this paper is to set up the Poincar\'{e}
inequality for monomial Gaussian weight and its improvements. More precisely,
let $d\mu_{A}=\frac{x^{A}e^{-\frac{1}{2}|x|^{2}}}{\int_{\mathbb{R}_{\ast}^{N}
}x^{A}e^{-\frac{1}{2}|x|^{2}}dx}dx$ be the Gaussian measure with monomial
weight. Let $X_{A}$ be the completion of $N_{\ast}$ under the norm $\left(
\int_{\mathbb{R}_{\ast}^{N}}\left\vert u\right\vert ^{2}d\mu_{A}\right)
^{\frac{1}{2}}+\left(  \int_{\mathbb{R}_{\ast}^{N}}\left\vert \nabla
u\right\vert ^{2}d\mu_{A}\right)  ^{\frac{1}{2}}$. Then we will show the
following Poincar\'{e} inequality with monomial weights:

\begin{theorem}
\label{T1}For all $u\in X_{A}$, we have
\begin{equation}
\int_{\mathbb{R}_{\ast}^{N}}\left\vert \nabla u\right\vert ^{2}d\mu_{A}%
\geq\int_{\mathbb{R}_{\ast}^{N}}\left\vert u-\int_{\mathbb{R}_{\ast}^{N}}%
ud\mu_{A}\right\vert ^{2}d\mu_{A}. \label{PGM}%
\end{equation}
Moreover, if $x^{A}$ is partial, then (\ref{PGM}) can be attained by
non-constant functions.
\end{theorem}

We note that this Poincar\'e inequality includes the classical Poincar\'e inequality with Gaussian measure by taking $\alpha_1=...=\alpha_N=0$.
The approach we use to prove (\ref{PGM}) is the Bakry-\'{E}mery
Curvature-Dimension criterion. This method was first introduced in the
1980's by Bakry and \'{E}mery in \cite{BE85}. Since then, it has been widely
used to study several problems such as heat-kernel and spectral estimates,
Harnack inequalities, Brunn--Minkowski-type inequalities, and isoperimetric,
functional and concentration inequalities, in different settings such as
(weighted) Riemannian geometry, Markov diffusion operators, metric measure
spaces, graphs and discrete spaces. The interested reader is referred to the
excellent book \cite{BGL14}, for instance, and the references therein.

As showed in \cite{CFLL24}, one can use the Poincar\'{e} inequality to
establish the stability of HUP. Therefore, we can expect that in order to
answer Question 1 and set up improved stability of HUP, we will need to study
improved and stability versions of the Poincar\'{e} inequality. This is indeed
our next aim. More precisely, we will show that by using the duality approach,
which is sometimes called the $L^{2}$ H\"{o}rmander method (see \cite{Bonn22}, for instance), we can obtain some
improvements for Theorem \ref{T1}. More clearly, we will prove that

\begin{theorem}
\label{T2}For $u\in X_{A}$, we have
\begin{align}
&  \int_{\mathbb{R}_{\ast}^{N}}\left\vert \nabla u\right\vert ^{2}d\mu
_{A}-\int_{\mathbb{R}_{\ast}^{N}}\left\vert u-\int_{\mathbb{R}_{\ast}^{N}%
}ud\mu_{A}\right\vert ^{2}d\mu_{A}\nonumber\\
&  \geq\frac{1}{2}\int_{\mathbb{R}_{\ast}^{N}}\left\vert \nabla\left[
u-\int_{\mathbb{R}_{\ast}^{N}}\left(  u-\int_{\mathbb{R}_{\ast}^{N}}ud\mu
_{A}\right)  xd\mu_{A}\cdot x\right]  \right\vert ^{2}d\mu_{A}. \label{IPM}%
\end{align}
If $x^{A}$ is partial, then (\ref{IPM}) can be attained by non-linear functions.
\end{theorem}

In the case when $x^{A}$ is partial, we can regard (\ref{IPM}) as a gradient
stability version of the Poincar\'{e} inequality (\ref{PGM}).

As a consequence, by applying the Poincar\'{e} inequality to the RHS of
(\ref{IPM}), we also obtain the following improved Poincar\'{e} inequality
with monomial Gaussian weight that can be used to study the improved stability
of the Heisenberg Uncertainty Principle with monomial weight:

\begin{proposition}
\label{P1}For $u\in X_{A}$, we have
\begin{align*}
&  \int_{\mathbb{R}_{\ast}^{N}}\left\vert \nabla u\right\vert ^{2}d\mu
_{A}-\int_{\mathbb{R}_{\ast}^{N}}\left\vert u-\int_{\mathbb{R}_{\ast}^{N}
}ud\mu_{A}\right\vert ^{2}d\mu_{A}\\
&  \geq\frac{1}{2}\int_{\mathbb{R}_{\ast}^{N}}\left\vert u-\int_{\mathbb{R}
_{\ast}^{N}}ud\mu_{A}+\int_{\mathbb{R}_{\ast}^{N}}uxd\mu_{A}\cdot
\int_{\mathbb{R}_{\ast}^{N}}xd\mu_{A}-\int_{\mathbb{R}_{\ast}^{N}}ud\mu
_{A}\left\vert \int_{\mathbb{R}_{\ast}^{N}}xd\mu_{A}\right\vert ^{2}\right. \\
&  \phantom{+++++}\left.  -\left(  \int_{\mathbb{R}_{\ast}^{N}}uxd\mu
_{A}\right)  \cdot x+\int_{\mathbb{R}_{\ast}^{N}}ud\mu_{A}\left(
\int_{\mathbb{R}_{\ast}^{N}}xd\mu_{A}\right)  \cdot x\right\vert ^{2}d\mu_{A}.
\end{align*}

\end{proposition}

In particular, in the classical Gaussian weight $d\mu_{0}=\frac{e^{-\frac
{1}{2}|x|^{2}}}{\int_{\mathbb{R}^{N}}e^{-\frac{1}{2}|x|^{2}}dx}dx=\frac
{e^{-\frac{1}{2}|x|^{2}}}{\left(  2\pi\right)  ^{\frac{N}{2}}}dx$, by noting
that $\int_{\mathbb{R}^{N}}xd\mu_{0}=\overrightarrow{0}$, we obtain as a
consequence of Theorem \ref{T2} and Proposition \ref{P1} that for $u\in X_{0}
$, we have
\begin{align}
&  \int_{\mathbb{R}^{N}}|\nabla u|^{2}d\mu_{0}-\int_{\mathbb{R}^{N}}\left\vert
u-\int_{\mathbb{R}^{N}}ud\mu_{0}\right\vert ^{2}d\mu_{0}\nonumber\\
&  \geq\frac{1}{2}\int_{\mathbb{R}^{N}}\left\vert \nabla u-\int_{\mathbb{R}
^{N}}uxd\mu_{0}\right\vert ^{2}d\mu_{0}\label{SPG}\\
&  \geq\frac{1}{2}\int_{\mathbb{R}^{N}}\left\vert u-\int_{\mathbb{R}^{N}}
ud\mu_{0}-\left(  \int_{\mathbb{R}^{N}}uxd\mu_{0}\right)  \cdot x\right\vert
^{2}d\mu_{0}. \label{IPG}%
\end{align}
Obviously, this is an improvement of the classical Poincar\'{e} inequality
with Gaussian weights:
\begin{equation}
\int_{\mathbb{R}^{N}}|\nabla u|^{2}d\mu_{0}\geq\int_{\mathbb{R}^{N}}\left\vert
u-\int_{\mathbb{R}^{N}}ud\mu_{0}\right\vert ^{2}d\mu_{0}. \label{PoiGau}%
\end{equation}
Moreover, (\ref{SPG}) implies that

\begin{theorem}
\label{P1.1}For $u\in X_{0}$, we have
\[
\int_{\mathbb{R}^{N}}|\nabla u|^{2}d\mu_{0}-\int_{\mathbb{R}^{N}}\left\vert
u-\int_{\mathbb{R}^{N}}ud\mu_{0}\right\vert ^{2}d\mu_{0}\geq\frac{1}{2}
\inf_{c,\overrightarrow{d}}\int_{\mathbb{R}_{\ast}^{N}}\left\vert
\nabla\left[  u-\left(  c+\overrightarrow{d}\cdot x\right)  \right]
\right\vert ^{2}d\mu_{0}.
\]
Also, the equality can be attained by non-linear functions.
\end{theorem}

This can be considered as a version of the gradient stability of the classical
Poincar\'{e} inequality with Gaussian measure.

We also note that as in Proposition \ref{P1.1}, the equality in (\ref{SPG})
can be achieved by non-linear functions. However, it is not the case for
(\ref{IPG}). In this situation, we will show that by using the standard
spectral analysis of the Hermite polynomials, we can obtain the following
version that provides the sharp constant for (\ref{IPG}):

\begin{theorem}
\label{P2}For $u\in X_{0}$, we have
\begin{align*}
&  \int_{\mathbb{R}^{N}}|\nabla u|^{2}d\mu_{0}-\int_{\mathbb{R}^{N}}\left\vert
u-\int_{\mathbb{R}^{N}}ud\mu_{0}\right\vert ^{2}d\mu_{0}\\
&  \geq\frac{1}{2}\int_{\mathbb{R}^{N}}\left\vert \nabla u-\int_{\mathbb{R}
^{N}}uxd\mu_{0}\right\vert ^{2}d\mu_{0}\\
&  \geq\int_{\mathbb{R}^{N}}\left\vert u-\int_{\mathbb{R}^{N}}ud\mu
_{0}-\left(  \int_{\mathbb{R}^{N}}uxd\mu_{0}\right)  \cdot x\right\vert
^{2}d\mu_{0}.
\end{align*}
Moreover, the equality can be attained by non-linear functions.
\end{theorem}

This result will play an important role in answering Question 1.

\medskip

Regarding the stability of the HUP with monomial weight, Theorem \ref{T1} and
its improvements are not enough to answer the Question 2. In fact, as showed
in \cite{CFLL24}, in order to study the stability of the HUP, one needs a
version of the Poincar\'{e} inequality with the monomial Gaussian weight
depending on the scaling factor. Therefore, our next goal is to establish a
scale-dependent Poincar\'{e} inequality with the monomial Gaussian weight.
More clearly, let $\lambda>0$ and $d\mu_{A,\lambda}=\frac{x^{A}e^{-\frac
{1}{2\left\vert \lambda\right\vert ^{2}}|x|^{2}}}{\int_{\mathbb{R}_{\ast}^{N}
}x^{A}e^{-\frac{1}{2\left\vert \lambda\right\vert ^{2}}|x|^{2}}dx}dx$ be the
Gaussian type measure with monomial weight. Let $X_{\lambda,A}$ be the
completion of $N_{\ast}$ under the norm $\left(  \int_{\mathbb{R}_{\ast}^{N}
}\left\vert u\right\vert ^{2}d\mu_{A,\lambda}\right)  ^{\frac{1}{2}}+\left(
\int_{\mathbb{R}_{\ast}^{N}}\left\vert \nabla u\right\vert ^{2}d\mu
_{A,\lambda}\right)  ^{\frac{1}{2}}$. Then we will establish the following
Poincar\'{e} inequality with $d\mu_{A,\lambda}$:

\begin{theorem}
\label{T3}For $u\in X_{\lambda,A}$, we have
\begin{equation}
\int_{\mathbb{R}_{\ast}^{N}}|\nabla u|^{2}d\mu_{A,\lambda}\geq\frac
{1}{\left\vert \lambda\right\vert ^{2}}\inf_{c}\int_{\mathbb{R}_{\ast}^{N}%
}\left\vert u-c\right\vert ^{2}d\mu_{A,\lambda}. \label{LPGM}%
\end{equation}
Moreover, if $x^{A}$ is partial, then (\ref{LPGM}) can be attained by
non-constant functions.
\end{theorem}

In the same spirit, we also obtain the improved scale-dependent Poincar\'{e}
inequality with the monomial Gaussian weight. More clearly, let $\lambda>0$
and recall that $d\mu_{A,\lambda}=\frac{x^{A}e^{-\frac{1}{2\left\vert
\lambda\right\vert ^{2}}|x|^{2}}}{\int_{\mathbb{R}_{\ast}^{N}}x^{A}
e^{-\frac{1}{2\left\vert \lambda\right\vert ^{2}}|x|^{2}}dx}dx$ is the
Gaussian type measure with monomial weight. Then we have that

\begin{theorem}
\label{T4}For $\lambda>0$ and $u\in X_{\lambda,A}$, we have
\begin{align*}
&  \int_{\mathbb{R}_{\ast}^{N}}|\nabla u|^{2}d\mu_{A,\lambda}\\
&  \geq\frac{1}{\left\vert \lambda\right\vert ^{2}}\inf_{c,\overrightarrow{d}
}\int_{\mathbb{R}_{\ast}^{N}}\left(  \left\vert u-c\right\vert ^{2}+\right. \\
&  \left.  +\frac{1}{2}\left\vert u-c+\overrightarrow{d}\cdot\int
_{\mathbb{R}_{\ast}^{N}}xd\mu_{A}-c\left\vert \int_{\mathbb{R}_{\ast}^{N}
}xd\mu_{A}\right\vert ^{2}-\overrightarrow{d}\cdot x+c\left(  \int
_{\mathbb{R}_{\ast}^{N}}xd\mu_{A}\right)  \cdot x\right\vert ^{2}\right)
d\mu_{A,\lambda}.
\end{align*}

\end{theorem}

Similarly, we can deduce the improved scale-dependent Poincar\'{e} inequality
with the Gaussian type measure. More precisely, let $d\mu_{\lambda}
=\frac{e^{-\frac{1}{2\left\vert \lambda\right\vert ^{2}}|x|^{2}}}
{\int_{\mathbb{R}^{N}}e^{-\frac{1}{2\left\vert \lambda\right\vert ^{2}}
|x|^{2}}dx}dx$ be the Gaussian type measure. Then we can establish the
following Poincar\'{e} inequality with $d\mu_{\lambda}$:

\begin{theorem}
\label{T5}For $\lambda>0$ and $u\in X_{\lambda,0}$, we have
\begin{align*}
&  \int_{\mathbb{R}^{N}}|\nabla u|^{2}d\mu_{\lambda}\\
&  \geq\frac{1}{\left\vert \lambda\right\vert ^{2}}\inf_{c,\overrightarrow{d}
}\int_{\mathbb{R}^{N}}\left(  \left\vert u-c\right\vert ^{2}+\left\vert
u-c-\overrightarrow{d}\cdot x\right\vert ^{2}\right)  d\mu_{\lambda}.
\end{align*}
Moreover, the equality can be attained by non-linear functions.
\end{theorem}

Our next goal is use the scale-dependent Poincar\'{e} inequality with the
monomial Gaussian weight to answer Question 2 and establish the stability of
the HUP with monomial weight, in the spirit of \cite{CFLL24}. Indeed, by
combining the identity in Proposition \ref{P0} and the scale-dependent
Poincar\'{e} inequality with the monomial Gaussian weight (\ref{LPGM}), we can
provide an affirmative answer to Question 2. More precisely, we will prove
that with the distance function
\[
d_{A}\left(  u,E_{HUPA}\right)  :=\inf_{c,\lambda\neq0}\left(  \int
_{\mathbb{R}_{\ast}^{N}}\left\vert u-ce^{-\frac{1}{2\lambda^{2}}|x|^{2}
}\right\vert ^{2}x^{A}dx\right)  ^{\frac{1}{2}},
\]
one has the following stability result for HUP with monomial weight:

\begin{theorem}
\label{T1.1}Let $u\in S_{A}$. Then
\begin{equation}
\delta_{A}\left(  u\right)  \geq d_{A}^{2}\left(  u,E_{HUPA}\right)  .
\label{SHM}%
\end{equation}
Moreover, if $x^{A}$ is partial, then (\ref{SHM}) can be attained by
nontrivial functions $u\notin E_{HUPA}$.
\end{theorem}
Here, we recall that we are using the deficit function
\[
\delta_{A}\left(  u\right)  =\left(  \int_{\mathbb{R}_{\ast}^{N}}|\nabla
u|^{2}x^{A}dx\right)  ^{\frac{1}{2}}\left(  \int_{\mathbb{R}_{\ast}^{N}
}|u|^{2}|x|^{2}x^{A}dx\right)  ^{\frac{1}{2}}-\frac{D}{2}\int_{\mathbb{R}
_{\ast}^{N}}|u|^{2}x^{A}dx.
\]

Moreover, we can use the improved scale-dependent Poincar\'{e} inequality to
obtain the improved stability of the HUP with monomial weight. More precisely,
let
\[
F:=\left\{  \left(  \alpha+\overrightarrow{\gamma}\cdot x\right)
e^{-\beta\left\vert x\right\vert ^{2}}:\alpha\in\mathbb{R}\text{,
}\overrightarrow{\gamma}\in\mathbb{R}^{N}\text{, }\beta>0\right\}
\]
and define
\[
\widetilde{d}_{A}\left(  u,F\right)  :=\inf_{c,\overrightarrow{d},\lambda
\neq0}\left(  \int_{\mathbb{R}_{\ast}^{N}}\left\vert u-\left(
c+\overrightarrow{d}\cdot x\right)  e^{-\frac{1}{2\lambda^{2}}|x|^{2}
}\right\vert ^{2}x^{A}dx\right)  ^{\frac{1}{2}}.
\]
Then we will prove the following improved stability of HUP with monomial weight:

\begin{theorem}
\label{T6}Let $u\in S_{A}$. Then
\[
\delta_{A}\left(  u\right)  -d_{A}^{2}\left(  u,E_{HUPA}\right)  \geq\frac
{1}{2}\widetilde{d}_{A}\left(  u,F\right)  .
\]

\end{theorem}

Finally, we will apply the refined Poincar\'{e} inequality (Theorem \ref{T5})
to provide an affirmative answer to Question 1. More precisely, recall that
$d_{1}(u,A)=\inf_{v\in A}\left\{  \left\Vert u-v\right\Vert _{2}\right\}  $
and define
\[
d_{2}(u,F):=\inf_{c,\overrightarrow{d},\lambda\neq0}\left(  \int
_{\mathbb{R}^{N}}\left\vert u-ce^{-\frac{1}{2\lambda^{2}}|x|^{2}}\right\vert
^{2}+\left\vert u-\left(  c+\overrightarrow{d}\cdot x\right)  e^{-\frac
{1}{2\lambda^{2}}|x|^{2}}\right\vert ^{2}dx\right)  ^{\frac{1}{2}}.
\]
Then we will prove the following result that can be considered as a stability
version of the sharp stability of the HUP established in \cite{CFLL24}:

\begin{theorem}
\label{T7}For all $u\in S_{0}:$
\[
\delta_{1}\left(  u\right)  \geq d_{2}^{2}(u,F).
\]
As a consequence
\[
\delta_{1}\left(  u\right)  -d_{1}^{2}(u,E_{HUP})\geq d_{1}^{2}(u,F).
\]

\end{theorem}

The paper is organized as follows: In subsection 2.1, we apply the method of
Curvature-Dimension condition to study the Poincar\'{e} inequality with
monomial Gaussian weight, and then use the duality approach to establish
its improvement and stability. In subsection 2.2, we use the Hermite spectral
method to derive sharp versions of the stability of the classical Poincar\'{e}
inequality with Gaussian weights. In subsection 2.3, we set up the
scale-dependent Poincar\'{e} inequality with Gaussian type measures. Finally,
in Section 3, we apply the Poincar\'{e} inequalities proved in Section 2 to
investigate several versions of the stability of the HUP. In particular, we
provide affirmative answers of Question 1 and Question 2 in Section 3.

\section{Poincar\'{e} type inequalities with Gaussian weights and the
stability}

\subsection{Poincar\'{e} inequality with monomial Gaussian weights and its
stability-Proofs of Theorem \ref{T1}, Theorem \ref{T2} and Proposition
\ref{P1}}

\phantom{???}

In this subsection, we will use the Bakry-\'{E}mery's $\Gamma$-calculus to establish the
Poincar\'{e} inequality with monomial Gaussian weight. For a very detailed study and various applications of the Bakry-\'{E}mery curvature-dimension criterion and $\Gamma$-calculus, we refer the interested reader to \cite{BGL14}.

Let $d\mu
_{A}=\frac{x^{A}e^{-\frac{1}{2}|x|^{2}}}{\int_{\mathbb{R}_{\ast}^{N} }
x^{A}e^{-\frac{1}{2}|x|^{2}}dx}dx$ be the Gaussian measure with monomial
weight. Then we have the following Poincar\'{e} inequality with
monomial-Gaussian measure $d\mu_{A}:$
\[
\int_{\mathbb{R}_{\ast}^{N}}|\nabla u|^{2}d\mu_{A}\geq\int_{\mathbb{R}_{\ast
}^{N}}\left\vert u-\int_{\mathbb{R}_{\ast}^{N}}ud\mu_{A}\right\vert ^{2}
d\mu_{A}.
\]

\begin{proof}
[Proof of Theorem \ref{T1}]Consider the diffusion operator
\[
\mathbf{L}_{A}:=\Delta-x\cdot\nabla+\Tilde{x}_{A}\cdot\nabla
\]
associated to the Gaussian measure with monomial weight $d\mu_{A}$. Here we
denote $\Tilde{x}_{A}=\left(  \frac{\alpha_{1}}{x_{1}},\frac{\alpha_{2}}
{x_{2}},\dots,\frac{\alpha_{N}}{x_{N}}\right)  $. Note that $-\mathbf{L}_{A}$
is symmetric and $\geq0$ with respect to the measure $\mu_{A}$ on $N_{\ast}$.
The Dirichlet form can be defined as follows:
\[
\mathcal{E}_{A}(u,v):=\int_{\mathbb{R}_{\ast}^{N}}\nabla u\cdot\nabla
v\ d\mu_{A}.
\]
By integration by part, we get
\begin{align*}
\mathcal{E}_{A}(u,v)  &  =\int_{\mathbb{R}_{\ast}^{N}}\nabla u\cdot\nabla
v\ d\mu_{A}\\
&  =-\int_{\mathbb{R}_{\ast}^{N}}u\operatorname{div}\left(  \nabla
v\ x^{A}e^{-\frac{1}{2}|x|^{2}}\right)  dx\\
&  =-\int_{\mathbb{R}_{\ast}^{N}}u\mathbf{L}_{A}vd\mu_{A}\\
&  =-\int_{\mathbb{R}_{\ast}^{N}}v\mathbf{L}_{A}ud\mu_{A}.
\end{align*}
This implies the invariance property of $\mathbf{L}_{A}:$
\[
\int_{\mathbb{R}_{\ast}^{N}}\mathbf{L}_{A}ud\mu_{A}=0\text{ }\forall u\in
N_{\ast}\text{.}
\]
We then can extend $-\mathbf{L}_{A}$ to a nonnegative self-adjoint operator on
$X_{A}$, which we still denote by $-\mathbf{L}_{A}$.

Next we define the associated \textquotedblleft carr\'{e} du
champ\textquotedblright\ operator $\Gamma_{A}:$
\[
\Gamma_{A}(u,v):=\frac{1}{2}\left[  \mathbf{L}_{A}\left(  uv\right)
-u\mathbf{L}_{A}v-v\mathbf{L}_{A}u\right]  .
\]
Note that
\[
\int_{\mathbb{R}_{\ast}^{N}}\Gamma_{A}(u,v)d\mu_{A}=-\int_{\mathbb{R}_{\ast
}^{N}}u\mathbf{L}_{A}vd\mu_{A}=\mathcal{E}_{A}(u,v).
\]
In particular,
\begin{align*}
\Gamma_{A}(u)  &  :=\Gamma_{A}(u,u)=\frac{1}{2}\left(  \mathbf{L}_{A}%
(u^{2})-u\mathbf{L}_{A}u-u\mathbf{L}_{A}u\right) \\
&  =\sum_{i}|\partial_{i}u|^{2}+{u\ \partial_{ii}u}+\sum_{j}{\frac{\alpha_{j}%
}{x_{j}}\ u\ \partial_{j}u}-\sum_{j}{x_{j}\ u\ \partial_{j}u}\\
&  -\sum_{i}{u\ \partial_{ii}u}+\sum_{i}{u\ x_{i}\partial_{i}u}-\sum
_{i}{u\ \frac{\alpha_{i}}{x_{i}}\partial_{i}u}\\
&  =\left\vert \nabla u\right\vert ^{2}.%
\end{align*}
We also define the \textquotedblleft carr\'{e} du champ
it\'{e}r\'{e}\textquotedblright\ operator $\Gamma_{2}$ by:
\[
\Gamma_{2}(u,v):=\frac{1}{2}\left(  \mathbf{L}_{A}\Gamma_{A}(u,v)-\Gamma
_{A}(u,\mathbf{L}_{A}v)-\Gamma_{A}(\mathbf{L}_{A}u,v)\right)  .
\]
Then,
\[
\int_{\mathbb{R}_{\ast}^{N}}\Gamma_{2}(u,v)d\mu_{A}=\int_{\mathbb{R}_{\ast
}^{N}}\mathbf{L}_{A}u\mathbf{L}_{A}vd\mu_{A}.
\]
Note that for all $u\in N_{\ast}:$
\begin{align*}
\Gamma_{2}(u)  &  :=\Gamma_{2}(u,u)=\frac{1}{2}\left(  \mathbf{L}_{A}%
\Gamma_{A}(u)-\Gamma_{A}(u,\mathbf{L}_{A}u)-\Gamma_{A}(\mathbf{L}%
_{A}u,u)\right) \\
&  =\sum_{i,j}|\partial_{ij}u|^{2}+{\partial_{j}u\partial_{jii}u}-\sum
_{i,j}{x_{j}\ \partial_{i}u\ \partial_{ij}u}+\sum_{i,j}{\frac{\alpha_{j}%
}{x_{j}}\ \partial_{i}u\ \partial_{ij}u}\\
&  -\sum_{i,j}{\partial_{j}u\ \partial_{iij}u}+\sum_{i}|\partial_{i}%
u|^{2}+\sum_{i,j}{x_{i}\ \partial_{j}u\ \partial_{ij}u}+\sum_{i}|\partial
_{i}u|^{2}\frac{\alpha_{i}}{x_{i}^{2}}-\sum_{i,j}{\frac{a_{i}}{x_{i}}%
\partial_{j}u\ \partial_{ij}u}\\
&  =||\nabla^{2}u||_{F}^{2}+|\nabla u|^{2}+\left\vert \widetilde{\nabla}%
_{A}u\right\vert ^{2}%
\end{align*}
where
\[
\widetilde{\nabla}_{A}u=\left(  \frac{\sqrt{\alpha_{1}}}{x_{1}}\partial
_{1}u,\dots,\frac{\sqrt{\alpha_{N}}}{x_{N}}\partial_{N}u\right)
\]
and%
\[
||A||_{F}=\sqrt{\sum_{i,j}|a_{ij}|^{2}}%
\]
is the Frobenius norm of the matrix $A$.

Since
\begin{align*}
\Gamma_{2}(u)  &  =||\nabla^{2}u||_{F}^{2}+|\nabla u|^{2}+\left\vert
\widetilde{\nabla}_{A}u\right\vert ^{2}\\
&  \geq\Gamma_{A}(u)=\left\vert \nabla u\right\vert ^{2},
\end{align*}
we have that the probability measure $\frac{x^{A}e^{-\frac{1}{2}|x|^{2}}}{\int
_{\mathbb{R}_{\ast}^{N}}x^{A}e^{-\frac{1}{2}|x|^{2}}dx}dx$ satisfies the
Curvature-Dimention condition $CD\left(  1,\infty\right)  $ \cite[Definition
1.16.1]{BGL14}. By \cite[Proposition 4.8.1]{BGL14}, we obtain the Poincar\'{e}
inequality with monomial-Gaussian measure $d\mu_{A}$ with constant $1$, that
is
\[
\int_{\mathbb{R}_{\ast}^{N}}|\nabla u|^{2}d\mu_{A}\geq\int_{\mathbb{R}_{\ast
}^{N}}\left\vert u-\int_{\mathbb{R}_{\ast}^{N}}ud\mu_{A}\right\vert ^{2}%
d\mu_{A}.
\]

Now, assume that $x^{A}$ is partial.\ WLOG, let $1\leq k<N$ and assume that
$x^{A}=x_{1}^{\alpha_{1}}...x_{k}^{\alpha_{k}}$, $\alpha_{i}\geq0$. Let
$u=a+\sum\limits_{j=k+1}^{N}a_{j}x_{j}$. Then since $\int\limits_{-\infty
}^{\infty}x_{j}e^{-\frac{1}{2}x_{j}^{2}}dx_{j}=0$, we have $\int
_{\mathbb{R}_{\ast}^{N}}ud\mu_{A}=a$. Also,%
\[
\int_{\mathbb{R}_{\ast}^{N}}\left\vert u-\int_{\mathbb{R}_{\ast}^{N}}ud\mu
_{A}\right\vert ^{2}d\mu_{A}=\sum\limits_{j=k+1}^{N}a_{j}^{2}%
\]
On the other hand, since $\nabla u=\sum\limits_{j=k+1}^{N}a_{j}\overrightarrow
{e}_{j}$, we have
\[
\int_{\mathbb{R}_{\ast}^{N}}\left\vert \nabla u\right\vert ^{2}d\mu_{A}%
=\sum\limits_{j=k+1}^{N}a_{j}^{2}.
\]
Therefore
\[
\int_{\mathbb{R}_{\ast}^{N}}\left\vert \nabla u\right\vert ^{2}d\mu_{A}%
=\int_{\mathbb{R}_{\ast}^{N}}\left\vert u-\int_{\mathbb{R}_{\ast}^{N}}%
ud\mu_{A}\right\vert ^{2}d\mu_{A}.
\]

\end{proof}

Next, by using the duality approach, we can prove the following improved
version of the Poincar\'{e} inequality with monomial-Gaussian measure
$d\mu_{A}$ :
\begin{align*}
&  \int_{\mathbb{R}_{\ast}^{N}}\left\vert \nabla u\right\vert ^{2}d\mu
_{A}-\int_{\mathbb{R}_{\ast}^{N}}\left\vert u-\int_{\mathbb{R}_{\ast}^{N}
}ud\mu_{A}\right\vert ^{2}d\mu_{A}\\
&  \geq\frac{1}{2}\int_{\mathbb{R}_{\ast}^{N}}\left\vert \nabla u-\int
_{\mathbb{R}_{\ast}^{N}}\left(  u-\int_{\mathbb{R}_{\ast}^{N}}ud\mu
_{A}\right)  xd\mu_{A}\right\vert ^{2}d\mu_{A}.
\end{align*}

\begin{proof}
[Proof of Theorem \ref{T2}]Let $u\in X_{A}$. It is obvious that $\int_{\mathbb{R}%
_{\ast}^{N}}\left(  u-\int_{\mathbb{R}_{\ast}^{N}}ud\mu_{A}\right)  d\mu
_{A}=0$. Let $w$ be the solution of the Poisson equation
\[
-\mathbf{L}_{A}w=u-\int_{\mathbb{R}_{\ast}^{N}}ud\mu_{A}.
\]
Note that $w$ satisfies the Neumann boundary condition: $\nabla w\cdot
\overrightarrow{\eta}=0\text{ on }\partial\mathbb{R}_{\ast}^{N}$, where $\overrightarrow{\eta}$ is the outer normal vector of $\mathbb{R}_{\ast
}^{N}$.
See \cite[Lemma 5.4]{Bonn22}, for instance.
Recall that
\[
\Gamma_{2}(w)=||\nabla^{2}w||_{F}^{2}+|\nabla w|^{2}+\left\vert \widetilde
{\nabla}_{A}w\right\vert ^{2}.
\]
Integrating both sides of the above identity, we obtain
\begin{align*}
&  \int_{\mathbb{R}_{\ast}^{N}}\left\vert u-\int_{\mathbb{R}_{\ast}^{N}}%
ud\mu_{A}\right\vert ^{2}d\mu_{A}=\int_{\mathbb{R}_{\ast}^{N}}\left\vert
\mathbf{L}_{A}w\right\vert ^{2}d\mu_{A}=\int_{\mathbb{R}_{\ast}^{N}}\Gamma
_{2}(w)d\mu_{A}\\
=  &  \int_{\mathbb{R}_{\ast}^{N}}|\nabla w|^{2}d\mu_{A}+\int_{\mathbb{R}%
_{\ast}^{N}}||\nabla^{2}w||_{F}^{2}d\mu_{A}+\int_{\mathbb{R}_{\ast}^{N}%
}\left\vert \widetilde{\nabla}_{A}w\right\vert ^{2}d\mu_{A}\\
=  &  \int_{\mathbb{R}_{\ast}^{N}}|\nabla w-\nabla u|^{2}d\mu_{A}%
+2\int_{\mathbb{R}_{\ast}^{N}}\nabla w\nabla ud\mu_{A}\\
&  -\int_{\mathbb{R}_{\ast}^{N}}|\nabla u|^{2}d\mu_{A}+\int_{\mathbb{R}_{\ast
}^{N}}||\nabla^{2}w||_{F}^{2}d\mu_{A}+\int_{\mathbb{R}_{\ast}^{N}}\left\vert
\widetilde{\nabla}_{A}w\right\vert ^{2}d\mu_{A}\\
=  &  \int_{\mathbb{R}_{\ast}^{N}}|\nabla w-\nabla u|^{2}d\mu_{A}%
-2\int_{\mathbb{R}_{\ast}^{N}}u\mathbf{L}_{A}wd\mu_{A}\\
&  -\int_{\mathbb{R}_{\ast}^{N}}|\nabla u|^{2}d\mu_{A}+\int_{\mathbb{R}_{\ast
}^{N}}||\nabla^{2}w||_{F}^{2}d\mu_{A}+\int_{\mathbb{R}_{\ast}^{N}}\left\vert
\widetilde{\nabla}_{A}w\right\vert ^{2}d\mu_{A}\\
=  &  \int_{\mathbb{R}_{\ast}^{N}}|\nabla w-\nabla u|^{2}d\mu_{A}%
+2\int_{\mathbb{R}_{\ast}^{N}}\left\vert u-\int_{\mathbb{R}_{\ast}^{N}}%
ud\mu_{A}\right\vert ^{2}d\mu_{A}\\
&  -\int_{\mathbb{R}_{\ast}^{N}}|\nabla u|^{2}d\mu_{A}+\int_{\mathbb{R}_{\ast
}^{N}}||\nabla^{2}w||_{F}^{2}d\mu_{A}+\int_{\mathbb{R}_{\ast}^{N}}\left\vert
\widetilde{\nabla}_{A}w\right\vert ^{2}d\mu_{A}.
\end{align*}
Therefore
\begin{align*}
&  \int_{\mathbb{R}_{\ast}^{N}}|\nabla u|^{2}d\mu_{A}-\int_{\mathbb{R}_{\ast
}^{N}}\left\vert u-\int_{\mathbb{R}_{\ast}^{N}}ud\mu_{A}\right\vert ^{2}%
d\mu_{A}\\
&  =\int_{\mathbb{R}_{\ast}^{N}}|\nabla w-\nabla u|^{2}d\mu_{A}+\int
_{\mathbb{R}_{\ast}^{N}}||\nabla^{2}w||_{F}^{2}d\mu_{A}+\int_{\mathbb{R}%
_{\ast}^{N}}\left\vert \widetilde{\nabla}_{A}w\right\vert ^{2}d\mu_{A}.
\end{align*}
In particular, we obtain the Poincar\'{e} inequality with monomial weights
\[
\int_{\mathbb{R}_{\ast}^{N}}|\nabla u|^{2}d\mu_{A}\geq\int_{\mathbb{R}_{\ast
}^{N}}\left\vert u-\int_{\mathbb{R}_{\ast}^{N}}ud\mu_{A}\right\vert ^{2}%
d\mu_{A}.
\]
Now, we note that
\[
\int_{\mathbb{R}_{\ast}^{N}}||\nabla^{2}w||_{F}^{2}d\mu_{A}=\sum_{i}%
\int_{\mathbb{R}_{\ast}^{N}}|\nabla\partial_{i}w|^{2}d\mu_{A}.
\]
By applying the Poincar\'{e} inequality with monomial Gaussian weight (Theorem
\ref{T1}), we get
\begin{align*}
\sum_{i}\int_{\mathbb{R}_{\ast}^{N}}|\nabla\partial_{i}w|^{2}d\mu_{A}  &
\geq\sum_{i}\int_{\mathbb{R}_{\ast}^{N}}\left\vert \partial_{i}w-\int
_{\mathbb{R}_{\ast}^{N}}\partial_{i}wd\mu_{A}\right\vert ^{2}d\mu_{A}\\
&  =\int_{\mathbb{R}_{\ast}^{N}}|\nabla w-\int_{\mathbb{R}_{\ast}^{N}}\nabla
wd\mu_{A}|^{2}d\mu_{A}.
\end{align*}
Also,
\begin{align*}
\int_{\mathbb{R}_{\ast}^{N}}\partial_{i}wd\mu_{A}  &  =-\int_{\mathbb{R}%
_{\ast}^{N}}w\left[  \frac{\alpha_{i}}{x_{i}}-x_{i}\right]  d\mu_{A}\\
&  =-\int_{\mathbb{R}_{\ast}^{N}}w\mathbf{L}_{A}x_{i}d\mu_{A}\\
&  =-\int_{\mathbb{R}_{\ast}^{N}}x_{i}\mathbf{L}_{A}wd\mu_{A}\\
&  =\int_{\mathbb{R}_{\ast}^{N}}x_{i}\left(  u-\int_{\mathbb{R}_{\ast}^{N}%
}ud\mu_{A}\right)  d\mu_{A}.
\end{align*}
Therefore,
\[
\int_{\mathbb{R}_{\ast}^{N}}||\nabla^{2}w||_{F}^{2}d\mu_{A}\geq\int
_{\mathbb{R}_{\ast}^{N}}\left\vert \nabla w-\int_{\mathbb{R}_{\ast}^{N}%
}\left(  u-\int_{\mathbb{R}_{\ast}^{N}}ud\mu_{A}\right)  xd\mu_{A}\right\vert
^{2}d\mu_{A}.%
\]
Hence
\begin{align*}
&  \int_{\mathbb{R}_{\ast}^{N}}|\nabla u|^{2}d\mu_{A}-\int_{\mathbb{R}_{\ast
}^{N}}\left\vert u-\int_{\mathbb{R}_{\ast}^{N}}ud\mu_{A}\right\vert ^{2}%
d\mu_{A}\\
=  &  \int_{\mathbb{R}_{\ast}^{N}}|\nabla w-\nabla u|^{2}d\mu_{A}%
+\int_{\mathbb{R}_{\ast}^{N}}||\nabla^{2}w||_{F}^{2}d\mu_{A}+\int
_{\mathbb{R}_{\ast}^{N}}\left\vert \widetilde{\nabla}_{A}w\right\vert ^{2}%
d\mu_{A}\\
\geq &  \int_{\mathbb{R}_{\ast}^{N}}|\nabla u-\nabla w|^{2}d\mu_{A}%
+\int_{\mathbb{R}_{\ast}^{N}}\left\vert \nabla w-\int_{\mathbb{R}_{\ast}^{N}%
}\left(  u-\int_{\mathbb{R}_{\ast}^{N}}ud\mu_{A}\right)  xd\mu_{A}\right\vert
^{2}d\mu_{A}\\
\geq &  \frac{1}{2}\int_{\mathbb{R}_{\ast}^{N}}\left\vert \nabla
u-\int_{\mathbb{R}_{\ast}^{N}}\left(  u-\int_{\mathbb{R}_{\ast}^{N}}ud\mu
_{A}\right)  xd\mu_{A}\right\vert ^{2}d\mu_{A}.
\end{align*}

Now, assume that $x^{A}$ is partial.\ WLOG, let $1\leq k<N$ and assume that
$x^{A}=x_{1}^{\alpha_{1}}...x_{k}^{\alpha_{k}}$, $\alpha_{i}\geq0$. Let
$u=a+\sum\limits_{j=k+1}^{N}a_{j}x_{j}+\sum\limits_{j=k+1}^{N}b_{j}\left(
x_{j}^{2}-1\right)  $. Then $\int_{\mathbb{R}_{\ast}^{N}}ud\mu_{A}=a$.
Therefore
\begin{align*}
&  \int_{\mathbb{R}_{\ast}^{N}}\left\vert u-\int_{\mathbb{R}_{\ast}^{N}}%
ud\mu_{A}\right\vert ^{2}d\mu_{A}\\
&  =\int_{\mathbb{R}_{\ast}^{N}}\left\vert \sum\limits_{j=k+1}^{N}a_{j}%
x_{j}+\sum\limits_{j=k+1}^{N}b_{j}\left(  x_{j}^{2}-1\right)  \right\vert
^{2}d\mu_{A}\\
&  =\sum\limits_{j=k+1}^{N}a_{j}^{2}+2b_{j}^{2}.
\end{align*}
On the other hand, since $\nabla u=\sum\limits_{j=k+1}^{N}\left(  a_{j}%
+2b_{j}x_{j}\right)  \overrightarrow{e}_{j}$, we have
\begin{align*}
\int_{\mathbb{R}_{\ast}^{N}}\left\vert \nabla u\right\vert ^{2}d\mu_{A} &
=\frac{1}{\int_{\mathbb{R}_{\ast}^{N}}x^{A}e^{-\frac{1}{2}|x|^{2}}dx}%
\int_{\mathbb{R}_{\ast}^{N}}\left\vert \sum\limits_{j=k+1}^{N}\left(
a_{j}+2b_{j}x_{j}\right)  \overrightarrow{e}_{j}\right\vert x_{1}^{\alpha_{1}%
}...x_{k}^{\alpha_{k}}e^{-\frac{1}{2}|x|^{2}}dx\\
&  =\sum\limits_{j=k+1}^{N}a_{j}^{2}+4b_{j}^{2}.
\end{align*}
That is,
\[
\int_{\mathbb{R}_{\ast}^{N}}\left\vert \nabla u\right\vert ^{2}d\mu_{A}%
-\int_{\mathbb{R}_{\ast}^{N}}\left\vert u-\int_{\mathbb{R}_{\ast}^{N}}%
ud\mu_{A}\right\vert ^{2}d\mu_{A}=\sum\limits_{j=k+1}^{N}2b_{j}^{2}.
\]
Also,
\begin{align*}
\int_{\mathbb{R}_{\ast}^{N}}\left(  u-\int_{\mathbb{R}_{\ast}^{N}}ud\mu
_{A}\right)  xd\mu_{A}\cdot x &  =\int_{\mathbb{R}_{\ast}^{N}}\left(
\sum\limits_{j=k+1}^{N}a_{j}x_{j}+\sum\limits_{j=k+1}^{N}b_{j}\left(
x_{j}^{2}-1\right)  \right)  xd\mu_{A}\cdot x\\
&  =\sum\limits_{j=k+1}^{N}a_{j}x_{j}.
\end{align*}
Hence
\[
u-\int_{\mathbb{R}_{\ast}^{N}}\left(  u-\int_{\mathbb{R}_{\ast}^{N}}ud\mu
_{A}\right)  xd\mu_{A}\cdot x=a+\sum\limits_{j=k+1}^{N}b_{j}\left(  x_{j}%
^{2}-1\right)
\]
and
\[
\nabla\left[  u-\int_{\mathbb{R}_{\ast}^{N}}\left(  u-\int_{\mathbb{R}_{\ast
}^{N}}ud\mu_{A}\right)  xd\mu_{A}\cdot x\right]  =\sum\limits_{j=k+1}%
^{N}2b_{j}x_{j}\overrightarrow{e}_{j}.
\]
That is
\begin{align*}
&  \int_{\mathbb{R}_{\ast}^{N}}\left\vert \nabla\left[  u-\int_{\mathbb{R}%
_{\ast}^{N}}\left(  u-\int_{\mathbb{R}_{\ast}^{N}}ud\mu_{A}\right)  xd\mu
_{A}\cdot x\right]  \right\vert ^{2}d\mu_{A}\\
&  =\int_{\mathbb{R}_{\ast}^{N}}\left\vert \sum\limits_{j=k+1}^{N}2b_{j}%
x_{j}\overrightarrow{e}_{j}\right\vert ^{2}d\mu_{A}\\
&  =\sum\limits_{j=k+1}^{N}4b_{j}^{2}.
\end{align*}
Therefore
\begin{align*}
&  \int_{\mathbb{R}_{\ast}^{N}}\left\vert \nabla u\right\vert ^{2}d\mu
_{A}-\int_{\mathbb{R}_{\ast}^{N}}\left\vert u-\int_{\mathbb{R}_{\ast}^{N}%
}ud\mu_{A}\right\vert ^{2}d\mu_{A}\\
&  =\frac{1}{2}\int_{\mathbb{R}_{\ast}^{N}}\left\vert \nabla\left[
u-\int_{\mathbb{R}_{\ast}^{N}}\left(  u-\int_{\mathbb{R}_{\ast}^{N}}ud\mu
_{A}\right)  xd\mu_{A}\cdot x\right]  \right\vert ^{2}d\mu_{A}.
\end{align*}

\end{proof}

As a consequence, we obtain
\begin{align*}
&  \int_{\mathbb{R}_{\ast}^{N}}\left\vert \nabla u\right\vert ^{2}d\mu
_{A}-\int_{\mathbb{R}_{\ast}^{N}}\left\vert u-\int_{\mathbb{R}_{\ast}^{N}
}ud\mu_{A}\right\vert ^{2}d\mu_{A}\\
&  \geq\frac{1}{2}\int_{\mathbb{R}_{\ast}^{N}}\left\vert u-\int_{\mathbb{R}
_{\ast}^{N}}ud\mu_{A}+\int_{\mathbb{R}_{\ast}^{N}}uxd\mu_{A}\cdot
\int_{\mathbb{R}_{\ast}^{N}}xd\mu_{A}-\int_{\mathbb{R}_{\ast}^{N}}ud\mu
_{A}\left\vert \int_{\mathbb{R}_{\ast}^{N}}xd\mu_{A}\right\vert ^{2}\right. \\
&  \phantom{+++++}\left.  -\left(  \int_{\mathbb{R}_{\ast}^{N}}uxd\mu
_{A}\right)  \cdot x+\int_{\mathbb{R}_{\ast}^{N}}ud\mu_{A}\left(
\int_{\mathbb{R}_{\ast}^{N}}xd\mu_{A}\right)  \cdot x\right\vert ^{2}d\mu_{A}.
\end{align*}

\begin{proof}
[Proof of Proposition \ref{P1}]By Theorem \ref{T2}, we have
\begin{align*}
&  \int_{\mathbb{R}_{\ast}^{N}}\left\vert \nabla u\right\vert ^{2}d\mu
_{A}-\int_{\mathbb{R}_{\ast}^{N}}\left\vert u-\int_{\mathbb{R}_{\ast}^{N}
}ud\mu_{A}\right\vert ^{2}d\mu_{A}\\
&  \geq\frac{1}{2}\int_{\mathbb{R}_{\ast}^{N}}\left\vert \nabla u-\int
_{\mathbb{R}_{\ast}^{N}}\left(  u-\int_{\mathbb{R}_{\ast}^{N}}ud\mu
_{A}\right)  xd\mu_{A}\right\vert ^{2}d\mu_{A}\\
&  =\frac{1}{2}\int_{\mathbb{R}_{\ast}^{N}}\left\vert \nabla\left(
u-\int_{\mathbb{R}_{\ast}^{N}}\left(  u-\int_{\mathbb{R}_{\ast}^{N}}ud\mu
_{A}\right)  xd\mu_{A}\cdot x\right)  \right\vert ^{2}d\mu_{A}.
\end{align*}
Now, by using the Poincar\'{e} inequality with monomial Gaussian weight
(Theorem \ref{T1}), we obtain
\begin{align*}
&  \frac{1}{2}\int_{\mathbb{R}_{\ast}^{N}}\left\vert \nabla\left(
u-\int_{\mathbb{R}_{\ast}^{N}}\left(  u-\int_{\mathbb{R}_{\ast}^{N}}ud\mu
_{A}\right)  xd\mu_{A}\cdot x\right)  \right\vert ^{2}d\mu_{A}\\
\geq &  \frac{1}{2}\int_{\mathbb{R}_{\ast}^{N}}\left\vert u-\int_{\mathbb{R}
_{\ast}^{N}}\left(  u-\int_{\mathbb{R}_{\ast}^{N}}ud\mu_{A}\right)  xd\mu
_{A}\cdot x-\int_{\mathbb{R}_{\ast}^{N}}\left(  u-\int_{\mathbb{R}_{\ast}^{N}
}\left(  u-\int_{\mathbb{R}_{\ast}^{N}}ud\mu_{A}\right)  xd\mu_{A}\cdot
x\right)  d\mu_{A}\right\vert ^{2}d\mu_{A}\\
=  &  \frac{1}{2}\int_{\mathbb{R}_{\ast}^{N}}\left\vert u-\int_{\mathbb{R}
_{\ast}^{N}}ud\mu_{A}+\int_{\mathbb{R}_{\ast}^{N}}uxd\mu_{A}\cdot
\int_{\mathbb{R}_{\ast}^{N}}xd\mu_{A}-\int_{\mathbb{R}_{\ast}^{N}}ud\mu
_{A}\left\vert \int_{\mathbb{R}_{\ast}^{N}}xd\mu_{A}\right\vert ^{2}\right. \\
&  \phantom{+++++}\left.  -\left(  \int_{\mathbb{R}_{\ast}^{N}}uxd\mu
_{A}\right)  \cdot x+\int_{\mathbb{R}_{\ast}^{N}}ud\mu_{A}\left(
\int_{\mathbb{R}_{\ast}^{N}}xd\mu_{A}\right)  \cdot x\right\vert ^{2}d\mu_{A}.
\end{align*}

\end{proof}

\subsection{Improved Poincar\'{e} inequality with the classical Gaussian
measure-Proofs of Theorem \ref{P1.1} and Theorem \ref{P2}}

\phantom{???}

In the classical Gaussian measure case $d\mu_{0}=\frac{e^{-\frac{1}{2}|x|^{2}
}}{\int_{\mathbb{R}^{N}}e^{-\frac{1}{2}|x|^{2}}dx}dx=\frac{e^{-\frac{1}
{2}|x|^{2}}}{\left(  2\pi\right)  ^{\frac{N}{2}}}dx$, note that $\int
_{\mathbb{R}_{\ast}^{N}}xd\mu_{0}=\overrightarrow{0}$, we can easily deduce
from Theorem \ref{T2} and Proposition \ref{P1} the following improved
Poincar\'{e} inequality with the classical Gaussian measure:

\begin{theorem}
\label{P2.1}For $u\in X_{0}$, we have
\begin{align*}
&  \int_{\mathbb{R}^{N}}|\nabla u|^{2}d\mu_{0}-\int_{\mathbb{R}^{N}}\left\vert
u-\int_{\mathbb{R}^{N}}ud\mu_{0}\right\vert ^{2}d\mu_{0}\\
&  \geq\frac{1}{2}\int_{\mathbb{R}^{N}}\left\vert \nabla u-\int_{\mathbb{R}
^{N}}xud\mu_{0}\right\vert ^{2}d\mu_{0}\\
&  \geq\frac{1}{2}\int_{\mathbb{R}^{N}}\left\vert u-\int_{\mathbb{R}^{N}}
ud\mu_{0}-\left(  \int_{\mathbb{R}^{N}}xud\mu_{0}\right)  \cdot x\right\vert
^{2}d\mu_{0}.
\end{align*}
Moreover, the first inequality can be attained by nonlinear functions.
\end{theorem}

We can also prove the above result by standard spectral analysis. For the convenience of the reader, we give here the proofs of Theorem \ref{P1.1} and Theorem \ref{P2} using the Hermite polynomials decomposition.

\begin{proof}
[Proof of Theorem \ref{P2.1} (Theorem \ref{P1.1})]It is well-known that the
spectrum $\sigma\left(  -\mathbf{L}_{0}\right)  =
\mathbb{N}
$ and the eigenfunctions are given by the Hermite polynomials $\left\{
\phi_{k}\right\}  _{k\geq0}$. We can assume that $\left\{  \phi_{k}\right\}
_{k\geq0}$ forms an orthonormal basis of $L^{2}\left(  \mu_{0}\right)  $ and
$-\mathbf{L}_{0}\phi_{k}=k\phi_{k}$. In particular, the first eigenvalue is
$\lambda_{0}=0$ with eigenfunction $\phi_{0}\left(  x\right)  =1$. The second
eigenvalue is $\lambda_{1}=1$ with eigenfunctions $\phi_{1,i}\left(  x\right)
=x_{i}$. The third eigenvalue is $\lambda_{2}=2$.

By Spectral theorem, we can write
\[
u=\sum_{k\geq0}c_{k}\phi_{k}
\]
with $c_{k}=\int_{\mathbb{R}^{N}}u\phi_{k}d\mu_{0}$. Note that $c_{0}
=\int_{\mathbb{R}^{N}}u\phi_{0}d\mu_{0}=\int_{\mathbb{R}^{N}}ud\mu_{0}$. Also,
$c_{1}\phi_{1}=\sum_{j}\left(  \int_{\mathbb{R}^{N}}ux_{j}d\mu_{0}\right)
x_{j}$.

Therefore
\begin{align*}
\int_{\mathbb{R}^{N}}\left\vert u-\int_{\mathbb{R}^{N}}ud\mu_{0}\right\vert
^{2}d\mu_{0}  &  =\int_{\mathbb{R}^{N}}\left(  \sum_{i\geq1}c_{i}\phi
_{i}\right)  \left(  \sum_{j\geq1}c_{j}\phi_{j}\right)  d\mu_{0}\\
&  =\sum_{i\geq1}c_{i}^{2}\text{.}%
\end{align*}
Also,
\[
-\mathbf{L}_{0}u=\sum_{k\geq0}kc_{k}\phi_{k}=\sum_{k\geq1}kc_{k}\phi_{k}.
\]
Therefore
\begin{align*}
\int_{\mathbb{R}^{N}}|\nabla u|^{2}d\mu_{0}  &  =-\int_{\mathbb{R}^{N}
}u\mathbf{L}_{0}ud\mu_{0}=\int_{\mathbb{R}^{N}}\left(  \sum_{k\geq0}kc_{k}
\phi_{k}\right)  \left(  \sum_{i\geq0}c_{i}\phi_{i}\right)  d\mu_{0}\\
&  =\sum_{i\geq1}ic_{i}^{2}.
\end{align*}
As a consequence, we obtain the Poincar\'{e} inequality:
\[
\int_{\mathbb{R}^{N}}|\nabla u|^{2}d\mu_{0}\geq\int_{\mathbb{R}^{N}}\left\vert
u-\int_{\mathbb{R}^{N}}ud\mu_{0}\right\vert ^{2}d\mu_{0}.
\]
Also,
\[
\int_{\mathbb{R}^{N}}|\nabla u|^{2}d\mu_{0}-\int_{\mathbb{R}^{N}}\left\vert
u-\int_{\mathbb{R}^{N}}ud\mu_{0}\right\vert ^{2}d\mu_{0}=\sum_{i\geq2}\left(
i-1\right)  c_{i}^{2}.
\]
Next, note that
\begin{align*}
&  \int_{\mathbb{R}^{N}}\left\vert \partial_{j}u-\int_{\mathbb{R}^{N}}
x_{j}ud\mu_{0}\right\vert ^{2}d\mu_{0}\\
&  =\int_{\mathbb{R}^{N}}\left\vert \partial_{j}u\right\vert ^{2}d\mu
_{0}-2\left(  \int_{\mathbb{R}^{N}}x_{j}ud\mu_{0}\right)  \left(
\int_{\mathbb{R}^{N}}\partial_{j}ud\mu_{0}\right)  +\left(  \int
_{\mathbb{R}^{N}}x_{j}ud\mu_{0}\right)  ^{2}\\
&  =\int_{\mathbb{R}^{N}}\left\vert \partial_{j}u\right\vert ^{2}d\mu
_{0}-2\left(  \int_{\mathbb{R}^{N}}x_{j}ud\mu_{0}\right)  \left(
\int_{\mathbb{R}^{N}}\partial_{j}u\partial_{j}x_{j}d\mu_{0}\right)  +\left(
\int_{\mathbb{R}^{N}}x_{j}ud\mu_{0}\right)  ^{2}\\
&  =\int_{\mathbb{R}^{N}}\left\vert \partial_{j}u\right\vert ^{2}d\mu
_{0}+2\left(  \int_{\mathbb{R}^{N}}x_{j}ud\mu_{0}\right)  \left(
\int_{\mathbb{R}^{N}}u\mathbf{L}_{0}x_{j}d\mu_{0}\right)  +\left(
\int_{\mathbb{R}^{N}}x_{j}ud\mu_{0}\right)  ^{2}\\
&  =\int_{\mathbb{R}^{N}}\left\vert \partial_{j}u\right\vert ^{2}d\mu
_{0}-\left(  \int_{\mathbb{R}^{N}}x_{j}ud\mu_{0}\right)  ^{2}.
\end{align*}
So
\begin{align*}
\int_{\mathbb{R}^{N}}\left\vert \nabla u-\int_{\mathbb{R}^{N}}xud\mu
_{0}\right\vert ^{2}d\mu_{0}  &  =\int_{\mathbb{R}^{N}}\left\vert \nabla
u\right\vert ^{2}d\mu_{0}-\sum_{i}\left(  \int_{\mathbb{R}^{N}}x_{j}ud\mu
_{0}\right)  ^{2}\\
&  =\sum_{i\geq2}ic_{i}^{2}.
\end{align*}
Therefore
\begin{align*}
&  \int_{\mathbb{R}^{N}}|\nabla u|^{2}d\mu_{0}-\int_{\mathbb{R}^{N}}\left\vert
u-\int_{\mathbb{R}^{N}}ud\mu_{0}\right\vert ^{2}d\mu_{0}\\
&  =\sum_{i\geq2}\left(  i-1\right)  c_{i}^{2}\\
&  \geq\frac{1}{2}\sum_{i\geq2}ic_{i}^{2}\\
&  =\frac{1}{2}\int_{\mathbb{R}^{N}}\left\vert \nabla u-\int_{\mathbb{R}^{N}
}xud\mu_{0}\right\vert ^{2}d\mu_{0}.
\end{align*}
Obviously, the equality happens with $u=\sum_{k=0}^{2}c_{k}\phi_{k}.$
\end{proof}

Actually, we can prove a better result that
\begin{align*}
&  \int_{\mathbb{R}^{N}}|\nabla u|^{2}d\mu_{0}-\int_{\mathbb{R}^{N}}\left\vert
u-\int_{\mathbb{R}^{N}}ud\mu_{0}\right\vert ^{2}d\mu_{0}\\
&  \geq\frac{1}{2}\int_{\mathbb{R}^{N}}\left\vert \nabla u-\int_{\mathbb{R}
^{N}}xud\mu_{0}\right\vert ^{2}d\mu_{0}\\
&  \geq\int_{\mathbb{R}^{N}}\left\vert u-\int_{\mathbb{R}^{N}}ud\mu
_{0}-\left(  \int_{\mathbb{R}^{N}}xud\mu_{0}\right)  \cdot x\right\vert
^{2}d\mu_{0}.
\end{align*}

\begin{proof}
[Proof of Theorem \ref{P2}]As above, let
\[
u=\sum_{k\geq0}c_{k}\phi_{k}.
\]
WLOG, assume that $c_{0}=\int_{\mathbb{R}^{N}}u\phi_{0}d\mu_{0}=\int
_{\mathbb{R}^{N}}ud\mu_{0}=0$. Then
\[
\int_{\mathbb{R}^{N}}|\nabla u|^{2}d\mu_{0}-\int_{\mathbb{R}^{N}}\left\vert
u-\int_{\mathbb{R}^{N}}ud\mu_{0}\right\vert ^{2}d\mu_{0}=\sum_{i\geq2}\left(
i-1\right)  c_{i}^{2}.
\]
Now
\begin{align*}
\int_{\mathbb{R}^{N}}\left\vert u-\left(  \int_{\mathbb{R}^{N}}xud\mu
_{0}\right)  \cdot x\right\vert ^{2}d\mu_{0}  &  =\int_{\mathbb{R}^{N}
}\left\vert u\right\vert ^{2}d\mu_{0}-\sum_{i}\left(  \int_{\mathbb{R}^{N}
}x_{j}ud\mu_{0}\right)  ^{2}\\
&  =\sum_{i\geq2}c_{i}^{2}%
\end{align*}
Therefore
\begin{align*}
&  \int_{\mathbb{R}^{N}}|\nabla u|^{2}d\mu_{0}-\int_{\mathbb{R}^{N}}\left\vert
u-\int_{\mathbb{R}^{N}}ud\mu_{0}\right\vert ^{2}d\mu_{0}\\
&  \geq\frac{1}{2}\int_{\mathbb{R}^{N}}\left\vert \nabla u-\int_{\mathbb{R}
^{N}}xud\mu_{0}\right\vert ^{2}d\mu_{0}\\
&  \geq\int_{\mathbb{R}^{N}}\left\vert u-\int_{\mathbb{R}^{N}}ud\mu
_{0}-\left(  \int_{\mathbb{R}^{N}}xud\mu_{0}\right)  \cdot x\right\vert
^{2}d\mu_{0}.
\end{align*}
Obviously, the equalities happen when $u=\sum_{k=0}^{2}c_{k}\phi_{k}.$
\end{proof}

\subsection{Scale-dependent Poincar\'{e} inequality with monomial Gaussian
weight-Proofs of Theorem \ref{T3}, Theorem \ref{T4} and Theorem \ref{T5}}

\phantom{???}

\begin{proof}
[Proof of Theorem \ref{T3}]We have from Theorem \ref{T1} that
\[
\int_{\mathbb{R}_{\ast}^{N}}|\nabla u|^{2}x^{A}e^{-\frac{1}{2}|x|^{2}}
dx\geq\int_{\mathbb{R}_{\ast}^{N}}\left\vert u-\frac{1}{\int_{\mathbb{R}
_{\ast}^{N}}x^{A}e^{-\frac{1}{2}|x|^{2}}dx}\int_{\mathbb{R}_{\ast}^{N}}
ux^{A}e^{-\frac{1}{2}|x|^{2}}dx\right\vert ^{2}x^{A}e^{-\frac{1}{2}|x|^{2}
}dx.
\]
Now, let $u\left(  x\right)  =v\left(  \lambda x\right)  $. Then $\nabla
u=\lambda\nabla v\left(  \lambda x\right)  $ and
\begin{align*}
\int_{\mathbb{R}_{\ast}^{N}}|\nabla u|^{2}x^{A}e^{-\frac{1}{2}|x|^{2}}dx  &
=\lambda^{2-N-D}\int_{\mathbb{R}_{\ast}^{N}}|\nabla v\left(  \lambda x\right)
|^{2}\left(  \lambda x\right)  ^{A}e^{-\frac{1}{2\left\vert \lambda\right\vert
^{2}}|\lambda x|^{2}}d\left(  \lambda x\right) \\
&  =\lambda^{2-N-D}\int_{\mathbb{R}_{\ast}^{N}}|\nabla v|^{2}x^{A}e^{-\frac
{1}{2\left\vert \lambda\right\vert ^{2}}|x|^{2}}dx
\end{align*}
\begin{align*}
\int_{\mathbb{R}_{\ast}^{N}}ux^{A}e^{-\frac{1}{2}|x|^{2}}dx  &  =\lambda
^{-N-D}\int_{\mathbb{R}_{\ast}^{N}}v\left(  \lambda x\right)  \left(  \lambda
x\right)  ^{A}e^{-\frac{1}{2\left\vert \lambda\right\vert ^{2}}|\lambda
x|^{2}}d\left(  \lambda x\right) \\
&  =\lambda^{-N-D}\int_{\mathbb{R}_{\ast}^{N}}vx^{A}e^{-\frac{1}{2\left\vert
\lambda\right\vert ^{2}}|x|^{2}}dx
\end{align*}

\[
\int_{\mathbb{R}_{\ast}^{N}}x^{A}e^{-\frac{1}{2}|x|^{2}}dx=\lambda^{-N-D}
\int_{\mathbb{R}_{\ast}^{N}}x^{A}e^{-\frac{1}{2\left\vert \lambda\right\vert
^{2}}|x|^{2}}dx
\]
\begin{align*}
&  \int_{\mathbb{R}_{\ast}^{N}}\left\vert u-\frac{1}{\int_{\mathbb{R}_{\ast
}^{N}}x^{A}e^{-\frac{1}{2}|x|^{2}}dx}\int_{\mathbb{R}_{\ast}^{N}}
ux^{A}e^{-\frac{1}{2}|x|^{2}}dx\right\vert ^{2}x^{A}e^{-\frac{1}{2}|x|^{2}
}dx\\
&  =\lambda^{-N-D}\int_{\mathbb{R}_{\ast}^{N}}\left\vert v-\frac{1}
{\int_{\mathbb{R}_{\ast}^{N}}x^{A}e^{-\frac{1}{2\left\vert \lambda\right\vert
^{2}}|x|^{2}}dx}\int_{\mathbb{R}_{\ast}^{N}}vx^{A}e^{-\frac{1}{2\left\vert
\lambda\right\vert ^{2}}|x|^{2}}dx\right\vert ^{2}x^{A}e^{-\frac
{1}{2\left\vert \lambda\right\vert ^{2}}|x|^{2}}dx
\end{align*}
Therefore, we obtain the following
\begin{align*}
&  \lambda^{2}\int_{\mathbb{R}_{\ast}^{N}}|\nabla v|^{2}x^{A}e^{-\frac
{1}{2\left\vert \lambda\right\vert ^{2}}|x|^{2}}dx\\
&  \geq\int_{\mathbb{R}_{\ast}^{N}}\left\vert v-\frac{1}{\int_{\mathbb{R}
_{\ast}^{N}}x^{A}e^{-\frac{1}{2\left\vert \lambda\right\vert ^{2}}|x|^{2}}
dx}\int_{\mathbb{R}_{\ast}^{N}}vx^{A}e^{-\frac{1}{2\left\vert \lambda
\right\vert ^{2}}|x|^{2}}dx\right\vert ^{2}x^{A}e^{-\frac{1}{2\left\vert
\lambda\right\vert ^{2}}|x|^{2}}dx\\
&  \geq\inf_{c}\int_{\mathbb{R}_{\ast}^{N}}\left\vert v-c\right\vert ^{2}
x^{A}e^{-\frac{1}{2\left\vert \lambda\right\vert ^{2}}|x|^{2}}dx.
\end{align*}

Now, assume that $x^{A}$ is partial.\ WLOG, let $1\leq k<N$ and assume that
$x^{A}=x_{1}^{\alpha_{1}}...x_{k}^{\alpha_{k}}$, $\alpha_{i}\geq0$. Then, as
in the proof of Theorem \ref{T1}, it is easy to verify that with
$v=a+\sum\limits_{j=k+1}^{N}a_{j}x_{j}$, then
\[
\int_{\mathbb{R}_{\ast}^{N}}|\nabla v|^{2}x^{A}e^{-\frac{1}{2\left\vert
\lambda\right\vert ^{2}}|x|^{2}}dx=\inf_{c}\int_{\mathbb{R}_{\ast}^{N}
}\left\vert v-c\right\vert ^{2}x^{A}e^{-\frac{1}{2\left\vert \lambda
\right\vert ^{2}}|x|^{2}}dx=\sum\limits_{j=k+1}^{N}a_{j}^{2}.
\]

\end{proof}

\begin{proof}
[Proof of Theorem \ref{T4}]From Proposition \ref{P1}, we have
\begin{align*}
&  \int_{\mathbb{R}_{\ast}^{N}}\left\vert \nabla u\right\vert ^{2}d\mu
_{A}-\int_{\mathbb{R}_{\ast}^{N}}\left\vert u-\int_{\mathbb{R}_{\ast}^{N}
}ud\mu_{A}\right\vert ^{2}d\mu_{A}\\
&  \geq\frac{1}{2}\int_{\mathbb{R}_{\ast}^{N}}\left\vert u-\int_{\mathbb{R}
_{\ast}^{N}}ud\mu_{A}+\int_{\mathbb{R}_{\ast}^{N}}uxd\mu_{A}\cdot
\int_{\mathbb{R}_{\ast}^{N}}xd\mu_{A}-\int_{\mathbb{R}_{\ast}^{N}}ud\mu
_{A}\left\vert \int_{\mathbb{R}_{\ast}^{N}}xd\mu_{A}\right\vert ^{2}\right. \\
&  \phantom{+++++}\left.  -\left(  \int_{\mathbb{R}_{\ast}^{N}}uxd\mu
_{A}\right)  \cdot x+\int_{\mathbb{R}_{\ast}^{N}}ud\mu_{A}\left(
\int_{\mathbb{R}_{\ast}^{N}}xd\mu_{A}\right)  \cdot x\right\vert ^{2}d\mu_{A}.
\end{align*}
As above, with the change of variable $u\left(  x\right)  =v\left(  \lambda
x\right)  $, we have
\begin{align*}
&  \int_{\mathbb{R}_{\ast}^{N}}|\nabla u|^{2}d\mu_{A}-\int_{\mathbb{R}_{\ast
}^{N}}\left\vert u-\int_{\mathbb{R}_{\ast}^{N}}ud\mu_{A}\right\vert ^{2}
d\mu_{A}\\
&  =\lambda^{2-N-D}\int_{\mathbb{R}_{\ast}^{N}}|\nabla v|^{2}x^{A}e^{-\frac
{1}{2\left\vert \lambda\right\vert ^{2}}|x|^{2}}dx\\
&  -\lambda^{-N-D}\int_{\mathbb{R}_{\ast}^{N}}\left\vert v-C_{A}\left(
\lambda,v\right)  \right\vert ^{2}x^{A}e^{-\frac{1}{2\left\vert \lambda
\right\vert ^{2}}|x|^{2}}dx
\end{align*}
for some $C_{A}\left(  \lambda,v\right)  $.

Also,
\begin{align*}
&  \int_{\mathbb{R}_{\ast}^{N}}\left\vert u-\int_{\mathbb{R}_{\ast}^{N}}
ud\mu_{A}+\int_{\mathbb{R}_{\ast}^{N}}uxd\mu_{A}\cdot\int_{\mathbb{R}_{\ast
}^{N}}xd\mu_{A}-\int_{\mathbb{R}_{\ast}^{N}}ud\mu_{A}\left\vert \int
_{\mathbb{R}_{\ast}^{N}}xd\mu_{A}\right\vert ^{2}\right. \\
&  \phantom{+++++}\left.  -\left(  \int_{\mathbb{R}_{\ast}^{N}}uxd\mu
_{A}\right)  \cdot x+\int_{\mathbb{R}_{\ast}^{N}}ud\mu_{A}\left(
\int_{\mathbb{R}_{\ast}^{N}}xd\mu_{A}\right)  \cdot x\right\vert ^{2}d\mu
_{A}\\
&  =\lambda^{-N-D}\int_{\mathbb{R}_{\ast}^{N}}\left\vert v-C_{A}\left(
\lambda,v\right)  +\overrightarrow{D}_{A}\left(  \lambda,v\right)  \cdot
\int_{\mathbb{R}_{\ast}^{N}}xd\mu_{A}-C_{A}\left(  \lambda,v\right)
\left\vert \int_{\mathbb{R}_{\ast}^{N}}xd\mu_{A}\right\vert ^{2}\right. \\
&  \phantom{+++++}\left.  -\overrightarrow{D}_{A}\left(  \lambda,v\right)
\cdot x+C_{A}\left(  \lambda,v\right)  \left(  \int_{\mathbb{R}_{\ast}^{N}
}xd\mu_{A}\right)  \cdot x\right\vert ^{2}x^{A}e^{-\frac{1}{2\left\vert
\lambda\right\vert ^{2}}|x|^{2}}dx
\end{align*}
for some vector $\overrightarrow{D}_{A}\left(  \lambda,v\right)  $. Therefore
\begin{align*}
&  \left\vert \lambda\right\vert ^{2}\int_{\mathbb{R}_{\ast}^{N}}|\nabla
v|^{2}x^{A}e^{-\frac{1}{2\left\vert \lambda\right\vert ^{2}}|x|^{2}}dx\\
&  \geq\inf_{c,\overrightarrow{d}}\int_{\mathbb{R}_{\ast}^{N}}\left(
\left\vert v-c\right\vert ^{2}+\right. \\
&  \left.  +\frac{1}{2}\left\vert v-c+\overrightarrow{d}\cdot\int
_{\mathbb{R}_{\ast}^{N}}xd\mu_{A}-c\left\vert \int_{\mathbb{R}_{\ast}^{N}
}xd\mu_{A}\right\vert ^{2}-\overrightarrow{d}\cdot x+c\left(  \int
_{\mathbb{R}_{\ast}^{N}}xd\mu_{A}\right)  \cdot x\right\vert ^{2}\right)
x^{A}e^{-\frac{1}{2\left\vert \lambda\right\vert ^{2}}|x|^{2}}dx.
\end{align*}

\end{proof}

\begin{proof}
[Proof of Theorem \ref{T5}]We have
\[
\int_{\mathbb{R}^{N}}|\nabla u|^{2}d\mu_{0}-\int_{\mathbb{R}^{N}}\left\vert
u-\int_{\mathbb{R}^{N}}ud\mu_{0}\right\vert ^{2}d\mu_{0}\geq\int
_{\mathbb{R}^{N}}\left\vert u-\int_{\mathbb{R}^{N}}ud\mu_{0}-\left(
\int_{\mathbb{R}^{N}}uxd\mu_{0}\right)  \cdot x\right\vert ^{2}d\mu_{0}.
\]
Let $u\left(  x\right)  =v\left(  \lambda x\right)  $. Then $\nabla
u=\lambda\nabla v\left(  \lambda x\right)  $ and so
\begin{align*}
&  \int_{\mathbb{R}^{N}}|\nabla u|^{2}d\mu_{0}-\int_{\mathbb{R}^{N}}\left\vert
u-\int_{\mathbb{R}^{N}}ud\mu_{0}\right\vert ^{2}d\mu_{0}\\
&  =\lambda^{2-N}\int_{\mathbb{R}^{N}}|\nabla v|^{2}e^{-\frac{1}{2\left\vert
\lambda\right\vert ^{2}}|x|^{2}}dx\\
&  -\lambda^{-N}\int_{\mathbb{R}^{N}}\left\vert v-C_{0}\left(  \lambda
,v\right)  \right\vert ^{2}e^{-\frac{1}{2\left\vert \lambda\right\vert ^{2}
}|x|^{2}}dx.
\end{align*}
for some $C_{0}\left(  \lambda,v\right)  $.

Finally
\begin{align*}
&  \int_{\mathbb{R}^{N}}\left\vert u-\int_{\mathbb{R}^{N}}ud\mu_{0}-\left(
\int_{\mathbb{R}^{N}}uxd\mu_{0}\right)  \cdot x\right\vert ^{2}d\mu_{0}\\
&  =\lambda^{-N}\int_{\mathbb{R}^{N}}\left\vert v-C_{0}\left(  \lambda
,v\right)  -\overrightarrow{D}_{0}\left(  \lambda,v\right)  \cdot x\right\vert
^{2}e^{-\frac{1}{2\left\vert \lambda\right\vert ^{2}}|x|^{2}}dx
\end{align*}
for some $C_{0}\left(  \lambda,v\right)  $ and $\overrightarrow{D}_{0}\left(
\lambda,v\right)  $. Therefore
\begin{align*}
&  \left\vert \lambda\right\vert ^{2}\int_{\mathbb{R}^{N}}|\nabla
v|^{2}e^{-\frac{1}{2\left\vert \lambda\right\vert ^{2}}|x|^{2}}dx\\
&  \geq\inf_{c,\overrightarrow{d}}\int_{\mathbb{R}^{N}}\left(  \left\vert
v-c\right\vert ^{2}+\left\vert v-c-\overrightarrow{d}\cdot x\right\vert
^{2}\right)  e^{-\frac{1}{2\left\vert \lambda\right\vert ^{2}}|x|^{2}}dx.
\end{align*}
Now, let $u_{\lambda}=a_{1}(x_{1}^{2}-\lambda^{2})+\dots+a_{N}(x_{N}%
-\lambda^{2})$. Then with $d\mu_{\lambda}=e^{-\frac{1}{2\left\vert
\lambda\right\vert ^{2}}|x|^{2}}dx$, we have by direct computation that%
\begin{align*}
\left\vert \lambda\right\vert ^{2}\int_{\mathbb{R}^{N}}|\nabla u_{\lambda
}|^{2}e^{-\frac{1}{2\left\vert \lambda\right\vert ^{2}}|x|^{2}}dx  &
=4\lambda^{2}\sum a_{i}^{2}\\
&  =\frac{1}{\lambda^{2}}\inf_{c,\Vec{d}}\left(  2(2\lambda^{4}\sum a_{i}%
^{2})+2\left\vert c\right\vert ^{2}+|d|^{2}\lambda^{2}\right) \\
&  =\inf_{c,\overrightarrow{d}}\int_{\mathbb{R}^{N}}\left(  \left\vert
u_{\lambda}-c\right\vert ^{2}+\left\vert u_{\lambda}-c-\overrightarrow{d}\cdot
x\right\vert ^{2}\right)  e^{-\frac{1}{2\left\vert \lambda\right\vert ^{2}%
}|x|^{2}}dx.
\end{align*}

\end{proof}

\section{Stability and the improved stability of the Heisenberg Uncertainty
Principle with monomial weight-Proofs of Theorem \ref{T1.1}, Theorem \ref{T6}
and Theorem \ref{T7}}

\begin{proof}
[Proof of Theorem \ref{T1.1}]Let $u\in N_{\ast}\setminus\left\{  0\right\}  $.
We have from Proposition \ref{P0} and the scale-dependent Poincar\'{e}
inequality for monomial Gaussian weight (Theorem \ref{T3}) with $\lambda
=\left(  \frac{\int_{\mathbb{R}_{\ast}^{N}}|u|^{2}|x|^{2}x^{A}dx}%
{\int_{\mathbb{R}_{\ast}^{N}}|\nabla u|^{2}x^{A}dx}\right)  ^{\frac{1}{4}}$
that
\begin{align*}
&  \left(  \int_{\mathbb{R}_{\ast}^{N}}|u|^{2}|x|^{2}x^{A}dx\right)
^{\frac{1}{2}}\left(  \int_{\mathbb{R}_{\ast}^{N}}|\nabla u|^{2}%
x^{A}dx\right)  ^{\frac{1}{2}}-\frac{D}{2}\int_{\mathbb{R}_{\ast}^{N}}%
|u|^{2}x^{A}dx\\
&  =\frac{\lambda^{2}}{2}\int_{\mathbb{R}_{\ast}^{N}}\left\vert \nabla\left(
ue^{\frac{|x|^{2}}{2\lambda^{2}}}\right)  \right\vert ^{2}e^{-\frac{|x|^{2}%
}{\lambda^{2}}}x^{A}dx\\
&  \geq\int_{\mathbb{R}_{\ast}^{N}}\left\vert ue^{\frac{|x|^{2}}{2\lambda^{2}%
}}-\frac{1}{\int_{\mathbb{R}_{\ast}^{N}}x^{A}e^{-\frac{1}{\lambda^{2}}|x|^{2}%
}dx}\int_{\mathbb{R}_{\ast}^{N}}ue^{\frac{|x|^{2}}{2\lambda^{2}}}%
x^{A}e^{-\frac{1}{\lambda^{2}}|x|^{2}}dx\right\vert ^{2}x^{A}e^{-\frac
{1}{\lambda^{2}}|x|^{2}}dx\\
&  =\int_{\mathbb{R}_{\ast}^{N}}\left\vert u-\frac{e^{-\frac{|x|^{2}}%
{2\lambda^{2}}}}{\int_{\mathbb{R}_{\ast}^{N}}x^{A}e^{-\frac{1}{\lambda^{2}%
}|x|^{2}}dx}\int_{\mathbb{R}_{\ast}^{N}}ue^{\frac{|x|^{2}}{2\lambda^{2}}}%
x^{A}e^{-\frac{1}{\lambda^{2}}|x|^{2}}dx\right\vert ^{2}x^{A}dx\\
&  \geq\inf_{c\text{,}\lambda\neq0}\int_{\mathbb{R}_{\ast}^{N}}\left\vert
u-ce^{-\frac{|x|^{2}}{2\lambda^{2}}}\right\vert ^{2}x^{A}dx.
\end{align*}
Now, assume that $x^{A}$ is partial.\ WLOG, let $1\leq k<N$ and assume that
$x^{A}=x_{1}^{\alpha_{1}}...x_{k}^{\alpha_{k}}$, $\alpha_{i}\geq0$. Let
$u=x_{N}e^{-\frac{|x|^{2}}{2}}\notin E_{HUPA}$. Then
\[
\int_{\mathbb{R}_{\ast}^{N}}|u|^{2}|x|^{2}x^{A}dx=\int_{\mathbb{R}_{\ast}^{N}%
}x_{N}^{2}|x|^{2}x^{A}e^{-|x|^{2}}dx
\]
Also since $\nabla u=\left(  \overrightarrow{e}_{N}-x_{N}x\right)
e^{-\frac{|x|^{2}}{2}}$, we have%
\begin{align*}
\int_{\mathbb{R}_{\ast}^{N}}|\nabla u|^{2}x^{A}dx  &  =\int_{\mathbb{R}_{\ast
}^{N}}\left\vert \overrightarrow{e}_{N}-x_{N}x\right\vert ^{2}x^{A}%
e^{-|x|^{2}}dx\\
&  =\int_{\mathbb{R}_{\ast}^{N}}\left(  1-2x_{N}^{2}+x_{N}^{2}\left\vert
x\right\vert ^{2}\right)  x^{A}e^{-|x|^{2}}dx\\
&  =\int_{\mathbb{R}_{\ast}^{N}}x_{N}^{2}\left\vert x\right\vert ^{2}%
x^{A}e^{-|x|^{2}}dx.
\end{align*}
Therefore
\[
\left(  \frac{\int_{\mathbb{R}_{\ast}^{N}}|u|^{2}|x|^{2}x^{A}dx}%
{\int_{\mathbb{R}_{\ast}^{N}}|\nabla u|^{2}x^{A}dx}\right)  ^{\frac{1}{4}}=1.
\]
By Proposition \ref{P0}, we get
\begin{align*}
&  \left(  \int_{\mathbb{R}_{\ast}^{N}}|\nabla u|^{2}x^{A}dx\right)
^{\frac{1}{2}}\left(  \int_{\mathbb{R}_{\ast}^{N}}|u|^{2}|x|^{2}%
x^{A}dx\right)  ^{\frac{1}{2}}-\frac{D}{2}\int_{\mathbb{R}_{\ast}^{N}}%
|u|^{2}x^{A}dx\\
&  =\frac{1}{2}\int_{\mathbb{R}_{\ast}^{N}}\left\vert \nabla\left(
ue^{\frac{|x|^{2}}{2}}\right)  \right\vert ^{2}x^{A}e^{-|x|^{2}}dx\\
&  =\frac{1}{2}\int_{\mathbb{R}_{\ast}^{N}}x^{A}e^{-|x|^{2}}dx.
\end{align*}
On the other hand,
\begin{align*}
&  \int_{\mathbb{R}_{\ast}^{N}}\left\vert u-ce^{-\frac{1}{2\lambda^{2}}%
|x|^{2}}\right\vert ^{2}x^{A}dx\\
&  =\int_{\mathbb{R}_{\ast}^{N}}\left\vert x_{N}e^{-\frac{|x|^{2}}{2}%
}-ce^{-\frac{1}{2\lambda^{2}}|x|^{2}}\right\vert ^{2}x^{A}dx\\
&  =\frac{1}{2}\int_{\mathbb{R}_{\ast}^{N}}x^{A}e^{-|x|^{2}}dx+c^{2}%
\int_{\mathbb{R}_{\ast}^{N}}e^{-\frac{1}{\lambda^{2}}|x|^{2}}x^{A}dx.
\end{align*}
Therefore
\[
d_{A}^{2}\left(  u,E_{HUPA}\right)  =\inf_{c,\lambda\neq0}\left(
\int_{\mathbb{R}_{\ast}^{N}}\left\vert u-ce^{-\frac{1}{2\lambda^{2}}|x|^{2}%
}\right\vert ^{2}x^{A}dx\right)  =\frac{1}{2}\int_{\mathbb{R}_{\ast}^{N}}%
x^{A}e^{-|x|^{2}}dx.
\]
That is
\[
\delta_{A}\left(  u\right)  =d_{A}^{2}\left(  u,E_{HUPA}\right)  .
\]

\end{proof}

\begin{proof}
[Proof of Theorem \ref{T6}]Let $u\in N_{\ast}\setminus\left\{  0\right\}  $.
We have from Proposition \ref{P0} and the improved scale-dependent
Poincar\'{e} inequality for monomial Gaussian weight (Theorem \ref{T4}) with
$\lambda=\left(  \frac{\int_{\mathbb{R}_{\ast}^{N}}|u|^{2}|x|^{2}x^{A}dx}
{\int_{\mathbb{R}_{\ast}^{N}}|\nabla u|^{2}x^{A}dx}\right)  ^{\frac{1}{4}}$
that
\begin{align*}
&  \left(  \int_{\mathbb{R}_{\ast}^{N}}|u|^{2}|x|^{2}x^{A}dx\right)
^{\frac{1}{2}}\left(  \int_{\mathbb{R}_{\ast}^{N}}|\nabla u|^{2}
x^{A}dx\right)  ^{\frac{1}{2}}-\frac{D}{2}\int_{\mathbb{R}_{\ast}^{N}}
|u|^{2}x^{A}dx\\
&  =\frac{\lambda^{2}}{2}\int_{\mathbb{R}_{\ast}^{N}}\left\vert \nabla\left(
ue^{\frac{|x|^{2}}{2\lambda^{2}}}\right)  \right\vert ^{2}e^{-\frac{|x|^{2}
}{\lambda^{2}}}x^{A}dx\\
&  \geq\inf_{c,\overrightarrow{d},\lambda\neq0}\int_{\mathbb{R}_{\ast}^{N}
}\left(  \left\vert u-ce^{-\frac{|x|^{2}}{2\lambda^{2}}}\right\vert ^{2}
+\frac{1}{2}\left\vert u-\left(  c+\overrightarrow{d}\cdot x\right)
e^{-\frac{1}{2\lambda^{2}}|x|^{2}}\right\vert ^{2}\right)  x^{A}dx\\
&  \geq\inf_{c\text{,}\lambda\neq0}\int_{\mathbb{R}_{\ast}^{N}}\left\vert
u-ce^{-\frac{|x|^{2}}{2\lambda^{2}}}\right\vert ^{2}x^{A}dx+\frac{1}{2}
\inf_{c,\overrightarrow{d},\lambda\neq0}\int_{\mathbb{R}_{\ast}^{N}}\left\vert
u-\left(  c+\overrightarrow{d}\cdot x\right)  e^{-\frac{1}{2\lambda^{2}
}|x|^{2}}\right\vert ^{2}x^{A}dx.
\end{align*}

\end{proof}

\begin{proof}
[Proof of Theorem \ref{T7}]By applying the improved scale-dependent
Poincar\'{e} inequality with the Gaussian type measure (Theorem \ref{T5}), we
obtain
\begin{align*}
&  \left(  \int_{\mathbb{R}^{N}}|u|^{2}|x|^{2}dx\right)  ^{\frac{1}{2}}\left(
\int_{\mathbb{R}^{N}}|\nabla u|^{2}dx\right)  ^{\frac{1}{2}}-\frac{N}{2}
\int_{\mathbb{R}^{N}}|u|^{2}dx\\
&  =\frac{\lambda^{2}}{2}\int_{\mathbb{R}^{N}}\left\vert \nabla\left(
ue^{\frac{|x|^{2}}{2\lambda^{2}}}\right)  \right\vert ^{2}e^{-\frac{|x|^{2}
}{\lambda^{2}}}dx\\
&  \geq\inf_{c,\overrightarrow{d},\lambda\neq0}\int_{\mathbb{R}^{N}}\left(
\left\vert ue^{\frac{|x|^{2}}{2\lambda^{2}}}-c\right\vert ^{2}+\left\vert
ue^{\frac{|x|^{2}}{2\lambda^{2}}}-c-\overrightarrow{d}\cdot x\right\vert
^{2}\right)  e^{-\frac{|x|^{2}}{\lambda^{2}}}dx\\
&  =\inf_{c,\overrightarrow{d},\lambda\neq0}\int_{\mathbb{R}^{N}}\left(
\left\vert u-ce^{-\frac{|x|^{2}}{2\lambda^{2}}}\right\vert ^{2}+\left\vert
u-\left(  c+\overrightarrow{d}\cdot x\right)  e^{-\frac{|x|^{2}}{2\lambda^{2}
}}\right\vert ^{2}\right)  dx\\
&  \geq\inf_{c,\lambda\neq0}\int_{\mathbb{R}^{N}}\left\vert u-ce^{-\frac
{|x|^{2}}{2\lambda^{2}}}\right\vert ^{2}dx+\inf_{c,\overrightarrow{d}
,\lambda\neq0}\int_{\mathbb{R}^{N}}\left\vert u-\left(  c+\overrightarrow
{d}\cdot x\right)  e^{-\frac{|x|^{2}}{2\lambda^{2}}}\right\vert ^{2}dx.
\end{align*}

\end{proof}

\end{document}